\documentclass[12pt,leqno]{amsart}
\usepackage{a4wide}
\usepackage{amssymb}
\usepackage{amsmath}
\usepackage{amsthm}
\usepackage{verbatim}
\usepackage{txfonts}
\usepackage{fullpage}
\usepackage[all,cmtip]{xy} 
\usepackage{yfonts}

\usepackage[T1]{fontenc}
\usepackage[utf8]{inputenc}
\usepackage{mathtools}   
\usepackage{amsfonts}

\numberwithin{equation}{section}

\theoremstyle{plain}

\newcommand{\A}{\ensuremath{{\mathbb{A}}}}
\newcommand{\C}{\ensuremath{{\mathbb{C}}}}
\newcommand{\Z}{\ensuremath{{\mathbb{Z}}}}
\renewcommand{\P}{\ensuremath{{\mathbb{P}}}}
\newcommand{\Q}{\ensuremath{{\mathbb{Q}}}}
\newcommand{\R}{\ensuremath{{\mathbb{R}}}}
\newcommand{\F}{\ensuremath{{\mathbb{F}}}}

\newcommand{\D}{\ensuremath{{\mathbb{D}}}}

\renewcommand{\H}{\ensuremath{{\mathbb{H}}}}

\newcommand{\N}{\ensuremath{{\mathbb{N}}}}
\newcommand{\charf}{\textbf{1}}

\newcommand{\GL}{\ensuremath{{\text{GL}}}}
\newcommand{\SL}{\ensuremath{{\text{SL}}}}
\newcommand{\Image}{\ensuremath{{\text{Image}}}}
\newtheorem{theo}{Theorem}[section]
\newtheorem{lem}[theo]{Lemma}
\newtheorem{prop}[theo]{Proposition}
\newtheorem{cor}[theo]{Corollary}

\theoremstyle{remark}
\newtheorem{rem}[theo]{Remark}

\theoremstyle{definition}
\newtheorem{defn}[theo]{Definition}

\newtheorem*{cor*}{Corollary}

\newcommand{\zxz}[4]{\begin{pmatrix} #1 & #2 \\ #3 & #4 \end{pmatrix}}
\renewcommand{\Re}{\operatorname{Re}}
\renewcommand{\Im}{\operatorname{Im}}
\newcommand{\Hom}{\operatorname{Hom}}

\title{Triple product formula and mass equidistribution on modular curves of level N}

\begin{document}
\author{Yueke Hu}

\address{Department of Mathematics, University of Wisconsin Madison, Van Vleck Hall, Madison, WI 53706, USA}
\email{yhu@math.wisc.edu}

\begin{abstract}
It was shown in previous works that the measure associated to holomorphic newforms of weight $k$ and level $q$ will tend weakly to the Haar measure on modular curve of level 1, as $qk\rightarrow \infty$.
In this paper we proved that this phenomenon is also true on modular curves of general level $N$. 
\end{abstract}
\maketitle

\section{introduction}
Let $\Gamma_0(N)$ be the standard congruence subgroup of $\SL_2(\Z)$, and let $Y_0(N)=\Gamma_0(N)\backslash \H$ be the corresponding modular curve of level $N$.
Let
\begin{equation}
 d\mu(z)=\frac{dxdy}{y^2}
\end{equation}
be the standard hyperbolic volume measure on $Y_0(N)$.

Let $f:\H\rightarrow \C$ be a  holomorphic newform of weight $k\in 2\N$ and level $q$, where $N|q$. For a bounded continuous test function $\phi$ on $Y_0(N)$, consider the following measure on $Y_0(N)$:
\begin{equation}
 \mu_f(\phi)=\int\limits_{\Gamma_0(q)\backslash\H}\phi(z)|f|^2(z)y^k \frac{dxdy}{y^2}.
\end{equation}
We will show that the measure $\mu_f$  converges weakly to $d\mu$ on $Y_0(N)$ as $qk\rightarrow \infty$. To be more precise, define 
\begin{equation}
 D_f(\phi)=\frac{\mu_f(\phi)}{\mu_f(1)}-\frac{\mu(\phi)}{\mu(1)}.
\end{equation}
\begin{theo}\label{thmmain}
 Let $\phi$ be a fixed bounded continuous function on $Y_0(N)$ and let $f$  traverse a sequence of holomorphic newforms of weight $k$ and level $q$,
 where $k\in 2\N$ and $N|q$. Then 
 \begin{equation}
  D_f(\phi)\rightarrow 0
 \end{equation}
whenever $qk\rightarrow \infty$.
\end{theo}
This is related to Rudnick and Sarnak's conjecture of quantum unique ergodicity for Laplacian eigenfunctions.
The case $N=q=1,k\rightarrow \infty$ was proved conditionally by Sarnak in \cite{Sa01} and by Luo-Sarnak in \cite{L&S03}. 
Holowinsky and Soundararajan proved this case unconditionally in \cite{Ho10} \cite{H&S10} and \cite{So10}. Their work provided basic framework for subsequent papers. 
Marshall proved the same result for Hilbert modular variety in \cite{Marshall}, that is, to replace rational field $\Q$ by a totally real number field.
Nelson in \cite{Ne11} and Nelson-Pitale-Saha in \cite{PPA14} generalized the work of Holowinsky-Soundararajan and solved the case $N=1, qk\rightarrow \infty$. This paper is a natural successor to these papers and allow general level $N$.

Theorem \ref{thmmain} will follow from the spectrum decomposition result (see Section \ref{subsecofspec} for more details) of square integrable functions on $Y_0(N)$, and the following two inequalities:
\begin{theo}\label{thmoffirstineq}
 Let $\phi$ be a Maass eigencuspform or incomplete Eisenstein series. Then
 \begin{equation}
 D_f(\phi)<<_{\phi,\epsilon}log(qk)^{\epsilon}\frac{(q/\sqrt{C})^{-1+2\alpha+\epsilon}}{log(kC)^\delta L(f, Ad, 1)}.
 \end{equation}
 Here $\alpha \leq 7/64$ is a bound towards the Ramanujan conjecture for $\phi$ at primes dividing $q$, and $\alpha=0$ if $\phi$ is incomplete Eisenstein series. 
  $C$ is the finite conductor of $\pi\times\pi$.
 $\delta=1/2$ or $1$ according as $\phi$ is cuspidal or incomplete Eisenstein series.
\end{theo}

\begin{theo}\label{thmofsecondineq}
  Let $\phi$ be a Maass eigencuspform or incomplete Eisenstein series. Then 
 \begin{equation}
 D_f(\phi)<<_{\phi,\epsilon}log(qk)^{\epsilon}q^\epsilon_\Diamond log(qk)^{1/12} L(f, Ad, 1)^{1/4}.
 \end{equation}
 Here $q_\Diamond$ is the largest integer such that $q_\Diamond^2|q$.
\end{theo}

The structure of this paper is organized to prove these two inequalities separately. Section 2 will be about notations and preliminary results. We will prove Theorem \ref{thmoffirstineq} for Maass eigencuspform in Section 3, and for incomplete Eisenstein series in Section 4. Section 5 will be devoted to prove Theorem \ref{thmofsecondineq}.

The idea to prove Theorem \ref{thmoffirstineq} is to adelize the integal $\mu_f(\phi)$ and reduce the problem to Triple product integral in case of Maass eigencuspform, or Rankin-Selberg integral in case of incomplete Eisenstein series. Then Theorem \ref{thmoffirstineq} would follow from Soundararajan's weak subconvexity bound and a reasonable bound for the local integrals. The main innovation of this paper is to control the local integrals for general high ramifications.

In the case of Triple product integral, the corresponding local integral is given by
\begin{equation}
 I_v=\int\limits_{\Q_p^*\backslash \GL_2(\Q_p)}\prod\limits_{i=1}^{3}<\pi_i(g)f_{i},f_{i}>dg.
\end{equation}
Here $f_i$ are the local new forms from local unitary representations $\pi_i$. $<\pi_i(g)f_{i},f_{i}>$ is the corresponding matrix coefficient for $\pi_i$. 

Woodbury in \cite{MW12} and Nelson in \cite{Ne11} computed the local integral for representations with squarefree levels.
In \cite{PPA14}, Nelson, Pitale and Saha computed $I_v$ for higher ramifications, with the assumption that $\pi_1=\pi_2$ and $\pi_3$ is unramified.
Their work is based on Lemma (3.4.2) of \cite{MV10}, which relates $I_v$ to the local Rankin-Selberg integral. But this method can not be generalized to the case when all the representations are supercuspidal, which is necessary for our purpose. 

In \cite{YH14}, we computed the $I_v$ in a more direct way, whenever one of the representations has higher level than the other two. 
The key tool is the descriptions of the Whittaker functional and matrix coefficient for highly ramified representations as developed in \cite{YH13} \cite{YH14}. 
Such results are not quite enough for the purpose of this paper as we have two representations with equally high level.
In this paper, we will further improve these descriptions and prove Proposition \ref{propWiofram} and Proposition \ref{propofsupportofMC}.

These results allow us to give a decent upper bound of $I_v$ for new forms in the cases when  
 $\pi_1$ and $\pi_2$ have the same high level and $\pi_3$ is also ramified. 
We will also use this tool to give an upper bound for local Rankin-Selberg integral in Proposition \ref{propRankinSelbergbound}.

On the other hand, it is relatively easier to generalize the proof in \cite{PPA14} for Theorem \ref{thmofsecondineq} to our case. The main difference is that there are now several cusps for $\Gamma_0(N)$, and one need to bound the Fourier coefficient of $\phi$ along each cusps. This is already done for Maass eigencuspform by Iwaniec in \cite{Iw}. We will deal with the case of Eisenstein series in Section 5.1. We believe such control is probably well-known or expected by experts. But as we didn't find a proper reference, we will give detailed proof in this paper.

I'd like to thank Simon Marshall and Tonghai Yang for helpful discussions and comments.

\section{Notations and preliminary results}
\subsection{Basic Definitions}
Let $\H$ be the upper half plane with the standard hyperbolic volume measure $d\mu=\frac{dxdy}{y^2}$. 
Let $\Delta=y^{-2}(\partial^2_x+\partial^2_y)$ be the hyperbolic Laplacian on $\H$.
Let $\GL_2^+$ be the subgroup of $\GL_2$ with positive determinants.
Then $\GL_2^+$ acts on $\H$ by fractional linear transformations. Let
\begin{equation}
 \Gamma_0(N)=\left\{\gamma\in\SL_2(\Z)|\gamma\equiv \zxz{*}{*}{0}{*}\mod{N}\right\}.
\end{equation}
Let
\begin{equation}
 \Gamma_\infty=\left\{\pm\zxz{1}{n}{0}{1}|n\in\Z\right\}.
\end{equation}
Let $Y_0(N)=\Gamma_0(N)\backslash\H$ be the modular curve of level $N$.

Given a function $f:\H\rightarrow \C$ and $\alpha=\zxz{a}{b}{c}{d}\in\GL_2^+$, we denote $f|_k\alpha$ to be the function
\begin{equation}\label{formulaofslashoperation}
 z\mapsto\det(\alpha)^{k/2}(cz+d)^{-k}f(\alpha z).
\end{equation}

A holomorphic cusp form of weight $k$ and level $q$ is a holomorphic function $f:\H\rightarrow\C$ that satisfies $f|_k\alpha=f$ for all $\alpha\in \Gamma_0(q)$ and vanishes at all cusps of $\Gamma_0(q)$. A holomorphic newform is a cusp form that is an eigenform
of the algebra of Hecke operators and orthogonal to the oldforms. (See \cite{DS05}.) 

A Maass cusp form $\phi$ of level $N$ (and weight 0)  is a $\Gamma_0(N)-$invariant eigenfunction of the hyperbolic Laplacian $\Delta$ on $\H$ that decays rapidly at the cusps of $\Gamma_0(N)$. 
\begin{equation}\label{Maasseigenvalue}
 (\Delta+1/4+r^2)\phi=0,\text{\ \ \ }r\in \R\cup i(-1/2,1/2).
\end{equation}

A Maass eigencuspform is a Maass cusp form which is an eigenfunction of the Hecke operators at all finite places and 
also the involution $T_{-1}:\phi\mapsto[z\mapsto\phi(-\bar{z})]$.

As we will care about asymptotic behaviors, we use the notation
\begin{equation}
 f(x,y)<<_{y}g(x,y)
\end{equation}
to indicate that there exists a positive real function $C(y)$ independent of $x$ such that
\begin{equation}
 |f(x,y)|\leq C(y)|g(x,y)|.
\end{equation}

Further if 
\begin{equation}
 f(x,y)<<_yg(x,y)<<_yf(x,y),
\end{equation}
we will say $f(x,y)\asymp_y g(x,y)$.

We shall also work adelically. In general let $\F$ be a number field. Let $v$ be a place of it and $\varpi_v$ be a local uniformizer at $v$. Let $O_v$ be the ring of integers of the local field $\F_v$.  For an integer $c$, let 
$K_0(\varpi_v^c)\subset\GL_2(O_v)$ be the set of matrices which are congruent to $\zxz{*}{*}{0}{*}\text{mod}{(\varpi_v^c)}$. Similarly let $K_1(\varpi_v^c)$ denote those congruent to $\zxz{*}{*}{0}{1}\text{mod}{(\varpi_v^c)}$.
In this paper we are mostly interested in the case when $\F=\Q$. In that case, we will let $p$ also denote both the place and the local uniformizer. But the arguments in Section 3 and 4 apply to general number field directly.

We will say that a local representation of $\GL_2$ at $p$ is of level $c$ if there is a unique up to constant element which is invariant under $K_1(p^c)$. Note that for representations of trivial central character this is equivalent to the invariance by $K_0(p^c)$.
For an automorphic representation of $\GL_2$ with trivial central character, we will say that it has finite conductor $N$ if its local component at $p$ is of level $e_p$, where $N=\prod\limits_{p|N}p^{e_p}$.

\subsection{Cusps and Fourier expansions}\label{subsecofcusps}
In general for a congruence subgroup $\Gamma\subset\SL_2(\Z)$,
denote $\mathcal{C}(\Gamma)=\Gamma\backslash \Q\cup\{\infty\}$ to be the set of cusps of $\Gamma\backslash\H$. Equivalently, 
\begin{equation}
 \mathcal{C}(\Gamma)=\Gamma\backslash\SL_2(\Z)/\Gamma_\infty.
\end{equation}
This is a finite set.
Fix $\mathfrak{a}\in\mathcal{C}(\Gamma)$, let $\tau_\mathfrak{a}\in\SL_2(\Z)$ be such that
\begin{equation}
 \tau_\mathfrak{a}\infty=\mathfrak{a}.
\end{equation}
Let $\Gamma_\mathfrak{a}$ be the stabilizer of $\mathfrak{a}$ in $\Gamma$. The width of the cusp $\mathfrak{a}$ is defined to be
\begin{equation}
 d_\mathfrak{a}=[\Gamma_\infty:\tau_\mathfrak{a}^{-1}\Gamma_\mathfrak{a}\tau_\mathfrak{a}].
\end{equation}
Define
\begin{equation}
 \sigma_\mathfrak{a}=\tau_\mathfrak{a}\zxz{d_\mathfrak{a}}{0}{0}{1}.
\end{equation}
It satisfies the property that
\begin{equation}
 \sigma_\mathfrak{a}^{-1}\Gamma_\mathfrak{a}\sigma_\mathfrak{a}=\Gamma_\infty.
\end{equation}

We now specify $\Gamma$ to be $\Gamma_0(N)$. As in Section 3.4.1 of \cite{PPA14}, one can consider the transitive right action of $\SL_2(\Z)$ on $\P^1(\Z/N)$:
\begin{equation}
 [x:y]\cdot\zxz{a}{b}{c}{d}=[ax+cy:bx+dy].
\end{equation}
Note that $\Gamma_0(N)$ is the stablizer of $[0:1]$ in $\SL_2(\Z)$. As a result,
the set of cusps $\mathcal{C}(\Gamma_0(N))$ can be parametrized by the set of ordered pairs
\begin{equation}
 \left\{[c:d]:\text{\ \ }c|N,d\in(\Z/(c,N/c))^*\right\}.
\end{equation}
If a cusp $\mathfrak{a}$ corresponds to a pair $[c:d]$ for $c|N$ and $d\in(\Z/(c,N/c))^*$, we will call $c_\mathfrak{a}=c$ the denominator of the cusp $\mathfrak{a}$. 
For a fixed $c|N$, let $\mathcal{C}[c]$ denote the set of cusps whose denominator is $c$.
Then by the above parametrization, 
\begin{equation}
\sharp\mathcal{C}[c]= \varphi((c,N/c)), 
\end{equation}
where $\varphi$ is the Euler totient function.

The width of a cusp can also be given in terms of $c_\mathfrak{a}$:
\begin{equation}
 d_\mathfrak{a}=\frac{[q,c_\mathfrak{a}^2]}{c_\mathfrak{a}^2}.
\end{equation}
Here $[x,y]$ means the least common multiple of $x$ and $y$.

Now let $f$ be a holomorphic new form of weight k and level $q$. Then for any cusp $\mathfrak{a}$, $f$ has Fourier expansion along $\mathfrak{a}$ in the following form:
\begin{equation}\label{formulaofholoFourier}
 f|_k\sigma_\mathfrak{a}(z)=y^{-k/2}\sum\limits_{n\in \N }\frac{\lambda_{f,\mathfrak{a}}(n)}{\sqrt{n}}\kappa_f(ny)e^{2\pi inx},
\end{equation}
where $\kappa_f(y)=y^{k/2}e^{-2\pi y}$ for $y$ positive real and $\lambda_{f,\mathfrak{a}}\in\C$.
From the Fourier expansion and Deligne's bound on the coefficients, one has the uniform control of growth of holomorphic new forms:
\begin{equation}\label{uniformgrowthcontrol}
|f(z)|<<e^{-2\pi y}.
\end{equation}

For a given $c|q$, define
\begin{equation}\label{defofaveragelambda}
 \lambda_{[c]}(n)=\left(\frac{1}{\varphi((c,q/c))}\sum\limits_{\mathfrak{a}\in\mathcal{C}[c]}|\lambda_\mathfrak{a}(n)|^2\right)^{1/2}.
\end{equation}
This average of Fourier coefficients is factorizable in the sense of \cite{PPA14}, and a bound of convolution sums for $\lambda_{[c]}(n)$ was also given there. 

For a Maass eigencuspform $\phi$ of level N and a fixed cusp $\mathfrak{a}$, there is a similar Fourier expansion
\begin{equation}\label{MaassFourierexpan}
 \phi(\sigma_\mathfrak{a} z)=\sum\limits_{n\neq 0}\frac{\lambda_{\phi,\mathfrak{a}}(n)}{\sqrt{n}}\kappa_{ir}(ny)e^{2\pi inx},
\end{equation}
where $\kappa_{ir}(y)=2y^{1/2}K_{ir}(2\pi y)$ with $K_{ir}$ being the standard K-Bessel function and $r$ is as in (\ref{Maasseigenvalue}). We have $|\kappa_{ir}(y)|\leq 1$ for all $s\in \R\cup i(-1/2,1/2)$ and all $y\in \R^+$.

As a corollary of Theorem 3.2 of \cite{Iw}, we have the following result 
\begin{cor}\label{corofboundFouriercuspidal}
For a Maass eigencuspform $\phi$ of level $N$ and any cusp $\mathfrak{a}$ as above, we have
\begin{equation}
\sum\limits_{|n|\leq M}|\lambda_{\phi,\mathfrak{a}}(n)|^2<<_{\phi}M.
\end{equation}
Using Cauchy-Schwartz inequality and that $\mathcal{C}(\Gamma)$ is finite,
\begin{equation}
\sum\limits_{|n|\leq M}\sum\limits_{\mathfrak{a}\in\mathcal{C}(\Gamma)}\frac{|\lambda_{\phi,\mathfrak{a}}(n)|}{\sqrt{|n|}}<<_{\phi,\epsilon}M^{1/2+\epsilon}.
\end{equation}
\end{cor}

\subsection{Eisenstein series and spectral theory of modular curve of level N}\label{subsecofspec}

For a compactly supported test function $h$ on $\R^+$ and a cusp $\mathfrak{a}$, the associated incomplete Eisenstein series for $\Gamma$ is defined to be
\begin{equation}
E_\mathfrak{a}(z,h)=\sum\limits_{\gamma\in \Gamma_\mathfrak{a}\backslash\Gamma}h(\Im (\sigma_{\mathfrak{a}}^{-1}\gamma z)).
\end{equation}
The Eisenstein series for $\Gamma$ along cusp $\mathfrak{a}$  is defined to be
\begin{equation}
E_\mathfrak{a}(z,s)=\sum\limits_{\gamma\in \Gamma_\mathfrak{a}\backslash\Gamma}\Im (\sigma_{\mathfrak{a}}^{-1}\gamma z)^s.
\end{equation}


\begin{theo}\label{thmofpole}
Let $\Gamma$ be a congruence subgroup and $\Re(s)\geq 1/2$. Then $E_\mathfrak{a}(z,s)$ has a unique pole at $s=1$ with residue
\begin{equation}
\text{res}_{s=1}E_\mathfrak{a}(z,s)=Vol(\Gamma\backslash\H)^{-1}.
\end{equation}
\end{theo}


According to \cite{Iw}, the space of square integrable functions on $\Gamma\backslash\H$ is spanned by the space of Maass cuspforms and the space of incomplete Eisenstein series; The latter can be further decomposed into residuals of Eisenstein series and direct integral of Eisenstein series at $\Re(s)=1/2$.

For incomplete Eisenstein series $E_\mathfrak{a}(z,h)$, one can get its spectrum decomposition by Mellin inversion formula. Let $\hat{h}(s)=\int\limits_{0}^{\infty}h(y)y^{-s-1}dy$ be the Mellin transform of $h$. It satisfies the growth control
\begin{equation}
\hat{h}(s)<<_{h,A}(1+|s|)^{-A}
\end{equation}
for any positive $A$.
Then the Mellin inversion formula claims that
\begin{equation}
h(y)=\frac{1}{2\pi i}\int\limits_{(2)}\hat{h}(s)y^s ds,
\end{equation}
where $\int\limits_{(2)}$ denotes the integral taken over the vertical contour from $2-i\infty$ to $2+i\infty$. Then by summing over $\sigma_\mathfrak{a}^{-1}\gamma$ translates for $\gamma\in \Gamma_\mathfrak{a}\backslash\Gamma$, we get
\begin{equation}
E_\mathfrak{a}(z,h)=\frac{1}{2\pi i}\int\limits_{(2)}\hat{h}(s)E_\mathfrak{a}(z,s)ds.
\end{equation}
Now move the integration to the line $\Re(s)=1/2$. By Cauchy's theorem, we have
\begin{equation}\label{spectrumdecompofincompleteEis}
E_\mathfrak{a}(z,h)=\frac{\hat{h}(1)}{Vol(\Gamma\backslash\H)}+
\frac{1}{2\pi i}\int\limits_{(1/2)}\hat{h}(s)E_\mathfrak{a}(z,s)ds.
\end{equation}
By a change of variable, we have in general for another cusp $\mathfrak{b}$,
\begin{equation}
E_\mathfrak{a}(\sigma_\mathfrak{b} z,h)=\frac{\hat{h}(1)}{Vol(\Gamma\backslash\H)}+
\frac{1}{2\pi i}\int\limits_{(1/2)}\hat{h}(s)E_\mathfrak{a}(\sigma_\mathfrak{b}z,s)ds.
\end{equation}
For the standard Eisenstein series $E(z,s)$, we have its Fourier expansion
\begin{equation}\label{FourierforE}
E(z,s)=y^s+M(s)y^{1-s}+\frac{1}{\xi(2s)}\sum\limits_{n\neq 0}\frac{\lambda_{s-1/2}}{\sqrt{|n|}}\kappa_{s-1/2}(ny)e^{2\pi i nx}.
\end{equation}
Here $\lambda_{s-1/2}(n)=\sum\limits_{ab=n}(a/b)^{s-1/2}$; $\kappa_{s-1/2}(y)=2|y|^{1/2}K_{s-1/2}(2\pi|y|)$ for $K$ the standard K-Bessel function; $M(s)=\xi(2s-1)/\xi(2s)$ and $\xi(s)=\pi^{-s/2}\Gamma(s/2)\zeta(s)$ is the completed Riemann zeta function.

In general let $\mathfrak{a}$, $\mathfrak{b}$ be two cusps for $\Gamma_0(N)$. Then by Theorem 3.4 of \cite{Iw},
\begin{equation}\label{formulaofclassicalFourier}
 E_\mathfrak{a}(\sigma_{\mathfrak{b}}z,s)=\delta_{\mathfrak{ab}}y^s+\varphi_\mathfrak{ab}(s)y^{1-s}+\sum\limits_{n\neq 0}\varphi_\mathfrak{ab}(n,s)\kappa_{s-1/2}(ny)e^{2\pi i nx},
\end{equation}
where $\delta_\mathfrak{ab}$ is the Kronecker symbol, $\varphi_\mathfrak{ab}(s)$ and $\varphi_\mathfrak{ab}(n,s)$ are defined using generalized Kloosterman sum.

Now let $\phi=E_\mathfrak{a}(\sigma_\mathfrak{b} z,h)$ be an incomplete Eisenstein series. Using its spectrum decomposition and the Fourier expansion above, we have
\begin{equation}
\phi(z)=\sum\limits_{n\in\Z}\phi_n(y)e^{2\pi i nx},
\end{equation}
where
\begin{equation}
\phi_n(y)=\frac{1}{2\pi i}\int\limits_{(1/2)}\hat{h}(s)\varphi_\mathfrak{ab}(n,s)\kappa_{s-1/2}(ny)ds,
\end{equation}
and
\begin{equation}\label{formulaEisconstantterm}
\phi_0(y)=\frac{\hat{h}(1)}{Vol(\Gamma\backslash\H)}+\frac{1}{2\pi i}\int\limits_{(1/2)}\hat{h}(s)(\delta_{\mathfrak{ab}}y^s+\varphi_\mathfrak{ab}(s)y^{1-s})ds.
\end{equation}

\subsection{Associate classical modular forms to automorphic forms}\label{subsecofAdelization}
Let $\phi$ be a modular form of weight $k$ and level $N$. We can associate to it an automorphic form $\tilde{\phi}$ as follows. By strong approximation,
\begin{equation}
\GL_2(\A)=\GL_2(\Q)\GL_2^+(\R)\prod\limits_{p}K_0(p^{e_p}).
\end{equation}
For any element $g\in\GL_2(\A)$, we can then write it as $$g=hg_\infty k,$$ where $h\in \GL_2(\Q)$, $g_\infty \in \GL_2^+(\R)$ and $k\in \prod\limits_{p}K_0(p^{e_p})$. Then define $\tilde{\phi}:\GL_2(\A)\rightarrow \C$ as
\begin{equation}
\tilde{\phi}(g)=\phi|_{g_\infty}(i),
\end{equation}
where $\phi|_{g_\infty}$ is as in (\ref{formulaofslashoperation}).

This automorphic form $\tilde{\phi}$ is called the adelization of $\phi$. It is clearly invariant under $\prod\limits_{p}K_0(p^{e_p})$, and its infinity component is weight k. We won't distinguish $\phi$ and $\tilde{\phi}$ later on when there is no confusion.

As an example, consider the Eisenstein series $E(z,s)$ of weight 0 and level 1. Its adelization is exactly 
\begin{equation}
E(g,s)=\sum\limits_{\gamma\in B(\Q)\backslash\GL_2(\Q)}\Phi_s(\gamma g),
\end{equation}
where $\Phi_s\in \text{Ind}_B^{\GL_2}(|\cdot|^{s-1/2},|\cdot|^{-s+1/2})$ satisfies
\begin{equation}
\Phi_s(\zxz{a}{b}{0}{d}g)=|a|^s|d|^{-s}\Phi_s(g).
\end{equation}
One can further translate the Fourier expansion of $E(z,s)$ into adelic languages.


Let $E(g,s)$ be the adelization of $E(z,s)$. For $z=x+iy$, let
\begin{equation}
g_z=(\zxz{y}{x}{0}{1},1,1,\cdots)\in\GL_2(\A),
\end{equation}
where the first component is the component at infinity.
So $E(z,s)$ can be recovered as 
\begin{equation}
E(z,s)=E(g_z,s).
\end{equation}
For Eisenstein series $E(g,s)$, we have its Fourier expansion
\begin{equation}\label{AdelicFourierforEis}
E(g_z,s)=\sum\limits_{a\in\Q}\int\limits_{\Q\backslash\A}E(\zxz{1}{x}{0}{1}g_z,s)\psi(-ax)dx.
\end{equation}
When $a=0$, the corresponding integral gives the constant term for the Fourier expansion. A standard unfolding technique would give (see \cite{Bump})
\begin{equation}\label{formulaofadelicEisconstantterm}
\int\limits_{\Q\backslash\A}E(\zxz{1}{x}{0}{1}g_z,s)dx=\Phi_s(g_z)+\int\limits_{\A}\Phi_s(\omega\zxz{1}{x}{0}{1}g_z)dx.
\end{equation}
Note that $\Phi_s(g_z)=y^s$, and the second integral gives $M(s)y^{1-s}$.

When $a\neq 0$, one can similarly get
\begin{equation}
\int\limits_{\Q\backslash\A}E(\zxz{1}{x}{0}{1}g,s)\psi(-ax)dx=W(\zxz{a}{0}{0}{1}g).
\end{equation}
Here $W$ is the global Whittaker functional associated to the Eisenstein series, defined as
\begin{equation}\label{formula2.4.1}
W(g)=\int\limits_{\A}\Phi_s(\omega\zxz{1}{x}{0}{1}g)\psi(-x)dx,
\end{equation}
where $\omega=\zxz{0}{1}{-1}{0}$.
The integral (\ref{formula2.4.1}) is directly a product of local integrals, which in turn are exactly the local Whittaker functional of induced representations.
We write 
\begin{equation}
W(g)=\prod\limits_{v}W_v(g),
\end{equation}
and we can normalize the local Whittaker functionals so that
\begin{equation}
W_\infty(\zxz{a}{0}{0}{1}g_z)= \frac{1}{\xi(2s)} \kappa_{s-1/2}(a y)e^{2\pi i a x},
\end{equation}
and
\begin{equation}
W_p(1)=1
\end{equation}
for all finite prime $p$.

If $\Phi_s$ is spherical at all finite places, then $W_p(\zxz{a}{0}{0}{1})$ is not zero only if $a$ is locally integral. As a result, the summation in the Fourier expansion (\ref{AdelicFourierforEis}) is actually only for integers. Then comparing (\ref{AdelicFourierforEis}) with (\ref{FourierforE}), we get
\begin{equation}
\lambda_{s-1/2}(n)=|n|^{1/2}\prod\limits_{p|n}W_p(\zxz{n}{0}{0}{1}).
\end{equation}
Under this identification, the fact that
\begin{equation}
\lambda_{s-1/2}(n)=\sum\limits_{ab=n}(\frac{a}{b})^{s-1/2}
\end{equation}
would follow directly from the well-known formula of the Whittaker function for spherical elements:
\begin{equation}\label{formulasphericalWhittaker}
W_p(\zxz{a}{0}{0}{1})=|a|^{1/2}\frac{p^{(v(a)+1)(s-1/2)}-p^{-(v(a)+1)(s-1/2)}}{p^{(s-1/2)}-p^{-(s-1/2)}},
\end{equation}
where $W_p$ is the Whittaker functional associated to the spherical element of $\text{Ind}_B^{\GL_2}(|\cdot|^{s-1/2},|\cdot|^{-s+1/2})$.

\subsection{Integrals at non-archimedean places}
Let $\psi$ an additive character of $\A$. Without loss of generality, we will always assume that $\psi$ is unramified at any finite place.
Let $O_v$ be the local ring of integers of $\F_v$. Let $p=|\varpi_v|_v^{-1}$. It can be a power of a prime.
\begin{lem}\label{lemofGaussint}
Let $m\in\F_v$ such that $v(m)=-j<0$, and $\nu$ be a character of $O_{v}^*$ of level $k>0$. Then
\begin{equation}
  |\int\limits_{v(x)=0}\psi_v(mx)\nu^{-1}(x)d^*x|=\begin{cases}
                                                   \sqrt{\frac{p}{(p-1)^2p^{k-1}}},&\text{\ if\ }j=k;\\
                                                   0,&\text{\ otherwise.}
                                                  \end{cases}
\end{equation}
\end{lem}
This is just a variant of the classical result on Gauss sum.

\begin{lem}\label{lemoflocalintofcharacters}
Let $\mu$ and $\nu$ both be multiplicative characters of $O_v^*$ and $j\in \Z$. Suppose $\mu$ is of level $i>0$. 
\begin{enumerate}
\item If $0<j\leq i-2$, then
\begin{equation*}
\int\limits_{v(x)=0}\mu(1+\varpi^j x)\nu(x)d^*x
\end{equation*}
is not zero only if $\nu$ is of level $i-j$. 
If $j=i-1$, then the above integral
is not zero only if $\nu$ is of level 0 or 1, and 
\begin{equation*}
\int\limits_{v(x)=0}\mu(1+ \varpi^{i-1} x)d^*x=-\frac{1}{p-1},
\end{equation*}
\item When $i>1$,
\begin{equation*}
\int\limits_{v(x)=0, x\notin -1+\varpi O_v }\mu(1+ x)\nu(x)d^*x
\end{equation*}
is not zero only if $\nu$ is of level $i$. When $i=1$, $\nu$ could  be of level 0 or 1, and
\begin{equation*}
\int\limits_{v(x)=0, x\notin -1+\varpi O_v }\mu(1+ x)d^*x=-\frac{1}{p-1},
\end{equation*}
\item
When $j<0$, 
\begin{equation*}
\int\limits_{v(x)=0}\mu(1+\varpi^j x)\nu(x)d^*x
\end{equation*}
is not zero only if $\nu$ is of level $i$.
\end{enumerate}

\end{lem}
\begin{proof}
If $\nu$ is of level greater than claimed, then the integral is zero by a simple change of variable. 
So we just need to show that when the level of $\nu$ is less than claimed, the integral is also zero. For conciseness we will only prove part (1), as the other two parts are very similar.  

In particular $0<j<i$ in this case. Suppose first $i-j\geq 2$ We can split the integral into intervals $a+\varpi^{i-j-1}O_v$ for $a\in (O_v/\varpi^{i-j-1}O_v)^*$. Then $\nu$ is constant on each such intervals by the condition of its level. On the other hand, let $x=a+\varpi^{i-j-1}u$, then
\begin{equation}
\mu(1+\varpi^j x)=\mu(1+\varpi^j a+\varpi^{i-1}u)=\mu(1+\varpi^j a)\mu(1+\frac{\varpi^{i-1}u}{1+\varpi^j a}),
\end{equation}
and by a change of variable,
\begin{equation}
\int\limits_{v(x)=0}\mu(1+\varpi^j x)\nu(x)d^*x=\frac{1}{(p-1)p^{i-j-2}}\sum\limits_{a\in (O_v/\varpi^{i-j-1}O_v)^*}\int\limits_{u\in O_v}\mu(1+\varpi^j a)\nu(a)\mu(1+\varpi^{i-1}u)du.
\end{equation}
The main observation here is that $\mu(1+\varpi^{i-1}u)$ as a function of $u$ is an additive character. So each integral in $u$ will give 0.

When $i-j=1$, let $\nu$ be of level $0<1$. Then $\mu(1+\varpi^{i-1} x)$ is additive of level $1$ in $x$, and
\begin{equation}
\int\limits_{v(x)=0}\mu(1+\varpi^{i-1} x)d^*x=-\frac{1}{p-1}.
\end{equation}
\end{proof}

Now we record some basic facts about integrals on $\GL_2(\F_v)$ when $v$ is finite. 
\begin{lem} \label{Iwasawadecomp}
For every positive integer $c$,
$$\GL_2(\F_v)=\coprod\limits_{0\leq i\leq c} B\zxz{1}{0}{\varpi_v^i}{1}K_1(\varpi_v^c).$$
Here $B$ is the Borel subgroup of $\GL_2$. 
\end{lem}
We normalize the Haar measure on $\GL_2(\F_v)$ such that $K_v$ has volume 1. 
Let $db$ be the left Haar measure on $\F_v^*\backslash B(\F_v)$, normalized such that $O_v^*\backslash B(O_v)$ has volume 1.
Then we have the following easy result (see, for example, \cite[Appendix A]{YH13}).
\begin{lem}\label{localintcoefficient}
Locally let $f$ be a $K_1(\varpi_v^c)-$invariant function, on which the center acts trivially. Then
 \begin{equation}
  \int\limits_{F_v^*\backslash\GL_2(\F_v)}f(g)dg=\sum\limits_{0\leq i\leq c}A_i\int\limits_{\F_v^*\backslash B(\F_v)}f(b\zxz{1}{0}{\varpi_v^i}{1})db,
 \end{equation}
where
$$A_0=\frac{p}{p+1}\text{,\ \ \ }A_c=\frac{1}{(p+1)p^{c-1}}\text{,\ \ \  and\ }A_i=\frac{p-1}{(p+1)p^i}\text{\ for\ }0<i<c.$$
\end{lem}

\subsection{Triple product formula}
Let $\pi_i$, $i=1,2,3$ be three unitary cuspidal automorphic representations with central characters $w_{\pi_i}$. Suppose that 
\begin{equation}
\prod_i w_{\pi_i}=1.
\end{equation}
Let $\Pi=\pi_1\otimes\pi_2\otimes\pi_3$. Then one can associate the triple product L-function $L(\Pi, s)$ to $\Pi$. See \cite{Garrett} and \cite{ps}. 

There exist local epsilon factors $\epsilon_v(\Pi_v,\psi_v,s)$ and global epsilon factor $\epsilon(\Pi,s)=\prod_v\epsilon(\Pi_v,\psi_v,s)$, such that,
\begin{equation}
 L(\Pi,1-s)=\epsilon(\Pi,s)L(\check{\Pi},s).
\end{equation}
With the assumption that $\prod_i w_{\pi_i}=1$, we have $$\Pi\cong\check{\Pi}.$$ The special values of local epsilon factors $\epsilon_v(\Pi_v,\psi_v,1/2)$ are actually independent of $\psi_v$ and always take value $\pm 1$. For simplicity, we will write 
$$\epsilon_v(\Pi_v,1/2)=\epsilon_v(\Pi_v,\psi_v,1/2).$$ For any place $v$, there is a unique (up to isomorphism) division algebra $\D_v$. Then
Prasad proved in \cite{Prasad} the following theorem about the dimension of the space of local trilinear forms:
\begin{theo}
\begin{enumerate}
 \item $\dim\Hom_{\GL_2(\F_v)}(\Pi_v,\C)\leq 1$, with the equality if and only if $\epsilon_v(\Pi_v,1/2)=1$.
 \item $\dim\Hom_{\D_v}(\Pi^{\D_v}_v,\C)\leq 1$, with the equality if and only if $\epsilon_v(\Pi_v,1/2)=-1$.
\end{enumerate}
Here $\Pi^{\D_v}_v$ is the image of $\Pi_v$ under Jacquet-Langlands correspondence.
\end{theo}
This motivated the following result which is conjectured by Jacquet and later on proved by Harris and Kudla in \cite{H&K91} and \cite{H&K04}:
\begin{theo}\label{thmofJacquetconj}
 $$\{L(\Pi,1/2)\neq 0\}  \Longleftrightarrow \left\{ \begin{array}{c}
                                                    \text{ there exist\ } \D\text{\ and\ } f_i\in \pi_i^{\D} \text{\ s.t.}\\
                                                    \int\limits_{Z_{\A}\D^*(\F)\backslash \D^*(\A)} f_1(g)f_2(g)f_3(g) dg\neq 0
                                                   \end{array}.\right\}
 $$
\end{theo}
Here the quaternion algebra $\D$ is uniquely determined by the local epsilon factors as in Prasad's criterion. This result hints that  $$\int\limits_{Z_{\A}\D^*(\F)\backslash \D^*(\A)} f_1(g)f_2(g)f_3(g) dg$$ could be a potential integral representation of special value of triple product L-function.
Later on there are a lot of work on explicitly relating both sides. In particular one can see Ichino's work in \cite{Ichino}. We only need a special version here .

\begin{equation}\label{Globaltriple}
 |\int\limits_{Z_{\A}\D^*(\F)\backslash \D^*(\A)} f_1(g)f_2(g)f_3(g) dg|^2=\frac{\zeta_\F^2(2)L(\Pi,1/2)}{8L(\Pi,Ad,1)}\prod_vI_v^0(f_{1,v},f_{2,v},f_{3,v}),
\end{equation}
where 
\begin{equation}\label{localtripleF}
I_v^0(f_{1,v},f_{2,v},f_{3,v})=\frac{L_v(\Pi_v,Ad,1)}{\zeta_v^2(2)L_v(\Pi_v,1/2)}I_v(f_{1,v},f_{2,v},f_{3,v}),
\end{equation}
and 
\begin{equation}
 I_v(f_{1,v},f_{2,v},f_{3,v})=\int\limits_{\F_v^*\backslash \D^*(\F_v)}\prod\limits_{i=1}^{3}<\pi^{\D}_i(g)f_{i,v},f_{i,v}>dg.
\end{equation}
We will however be mainly interested in the case when $\D$ is the matrix algebra. If $\D$ is a division algebra, the integral $\mu_f(\phi)$ will be zero automatically.

\subsection{Whittaker model for some highly ramified representations}
This subsection is purely local, so we will suppress the subscript $v$ for all notations.

Let $\pi$ be a local irreducible (generic) representation of $\GL_2$. Let $\psi$ be a fixed unramified additive character.
Then there is a unique realization of $\pi$ in the space of functions $W$ on $\GL_2$ such that
\begin{equation}
W(\zxz{1}{n}{0}{1}g)=\psi(n)W(g).
\end{equation}

When $\pi$ is unitary, one can define a unitary pairing on $\pi$ using the Whittaker model:
\begin{equation}\label{localnormalization2}
<W_1,W_2>=\int_{\F^*}W_1(\zxz{\alpha}{0}{0}{1})\overline{W_2(\zxz{\alpha}{0}{0}{1})}d^*\alpha.
\end{equation}

Now let $\pi$ be a supercuspidal representation.
The Kirillov model of $\pi$ is a unique realization on the space of Schwartz functions $S(\F ^*)$ such that
\begin{equation}\label{Kirilmodel}
\pi(\zxz{a_1}{m}{0}{a_2})\varphi(x)=w_{\pi}(a_2)\psi(ma_2^{-1}x)\varphi(a_1a_2^{-1}x),
\end{equation}
where $w_{\pi}$ is the central character for $\pi$. 
Let $W_\varphi$ be the Whittaker function associated to $\varphi$. Then they are related by
$$\varphi(\alpha)=W_\varphi(\zxz{\alpha}{0}{0}{1}),$$
$$W_\varphi(g)=\pi(g)\varphi(1).$$
When $\pi$ is unitary, one can define the $G-$invariant unitary pairing on Kirillov model by
\begin{equation}
<f_1,f_2>=\int\limits_{\F ^*}f_1(x)\overline{f_2}(x)d^*x.
\end{equation}
This is consistent with the unitary pairing defined above using Whittaker model.

By Bruhat decomposition, one just has to know the action of $\omega=\zxz{0}{1}{-1}{0}$ to understand the whole group action. 

For $\nu$ a character of $O_v^*$,
define $$\charf_{\nu,n}(x)=\begin{cases}
                        \nu(u), &\text{if\ } x=u\varpi^n\text{\  for\ } u\in O_v^*;\\
						0,&\text{otherwise}.
                       \end{cases}
 $$ Roughly speaking, it's the character $\nu$ supported at $v(x)=n$.
We can then describe the action of $\omega=\zxz{0}{1}{-1}{0}$ on $\charf_{\nu,n}$ explicitly according to \cite{JL70}:

\begin{equation}\label{singleaction}
\pi(\omega)\charf_{\nu,n}=C_{\nu w_0^{-1}}z_0^{-n}\charf_{\nu^{-1}w_0,-n+n_{\nu^{-1}}}.
\end{equation}
Here $z_0=w_\pi(\varpi)$ and $w_0=w_\pi|_{O_F^*}$. $n_\nu$ is an integer decided by the representation $\pi$ and the character $\nu$ (and independent of $n$).
It's well-known that $n_{\nu}\leq -2$ for any $\nu$.
When we pick $\nu$ to be the trivial character, the number $-n_1$ is actually the level of this supercuspidal representation. Denote $c=-n_1$. 
The  local new form is simply $\charf_{1,0}$.

The relation $\omega^2=-\zxz{1}{0}{0}{1}$ implies 
\begin{equation}
n_{\nu}=n_{\nu^{-1}w_0^{-1}},\text{\ \ } C_\nu C_{\nu^{-1}w_0^{-1}}=w_0(-1)z_0^{n_\nu}.
\end{equation}

When $\pi$ is unitary, we have
\begin{equation}\label{absolutevalueofC}
 |C_\nu|=1.
\end{equation}
One can easily show this by using the fact that 
\begin{equation}
 <\pi(\omega)\charf_{\nu,0},\pi(\omega)\charf_{\nu,0}>=<\charf_{\nu,0},\charf_{\nu,0}>.
\end{equation}

It is essentially proved in \cite[Proposition B.3]{YH13} that

\begin{prop}\label{Propofnmu}
Suppose that $c=-n_1\geq 2$ is the level of a supercuspidal representation $\pi$ whose central character is unramified or level 1. If $p\neq 2$ and $\nu$ is a level $i$ character, then we have $$n_\nu=\min\{-c,-2i\}.$$ 
When $p=2$ or the central character of $\pi$ is highly ramified, we have the same statement, except when $c\geq 4$ is an even integer and $i=c/2$. In that case, we only claim $n_\nu\geq -c$.
\end{prop}


Now let $\pi$ be a unitary induced representation $\pi(\mu_1,\mu_2)$, where $\mu_1$ and $\mu_2$ are both ramified of level $k$. Let $c=2k$ be the level of $\pi$. Then by the classical results in \cite{Ca73}, there exists a new form in the model of induced representation, which is right $K_1(\varpi^c)-$invariant and supported on
$$B\zxz{1}{0}{\varpi^{k}}{1}K_1(\varpi^c),$$
where $B$ is the Borel subgroup.
From now on let $W$ be the Whittaker function associated to this new form.

Locally for an induced representation of $\GL_2$, one can compute its Whittaker functional by the following formula:
\begin{equation}\label{computeWhit}
W (g)=\int\limits_{m\in \F }\varphi(\omega\zxz{1}{m}{0}{1}g)\psi(-m)dm,
\end{equation}
where $\varphi$ is an element of $\pi$ in the model of induced representation and $\omega$ is the matrix $\zxz{0}{1}{-1}{0}$.

\begin{defn}\label{defofWi}
 Denote 
\begin{equation}\label{Wi}
 W^{(i)}(\alpha)=W (\zxz{\alpha}{0}{0}{1}\zxz{1}{0}{\varpi^i}{1}).
\end{equation}
\end{defn}
Let 
\begin{equation}
C_0=\int\limits_{u\in O_F^*} \mu_1(-\varpi^{k})\mu_2(-\varpi^{-k} u)\psi(-\varpi^{-k} u)du.
\end{equation}

In \cite{YH14} we gave the following formulae to compute $W^{(i)}$:

If $i<k$, then 
\begin{equation}
W^{(i)}(\alpha)=C_0^{-1}\int\limits_{u\in O_F^*}\mu_1(-\frac{\varpi^i}{u})\mu_2(\alpha\varpi^{-i}(1-\varpi^{k-i}u))\psi(\alpha\varpi^{-i}(1-\varpi^{k-i}u))p^{i-k-v(\alpha)/2}du.
\end{equation}

If $k<i\leq c$, then 
\begin{equation}
W^{(i)}(\alpha)=C_0^{-1}\int\limits_{u\in O_F^*} \mu_1(-\frac{\varpi^{k}}{1+u\varpi^{i-k}})\mu_2(-\varpi^{-k}\alpha u)p^{-v(\alpha)/2}\psi(-\varpi^{-k}\alpha u)du.
\end{equation}

If $i=k$, 
\begin{equation}
W^{(k)}(\alpha)=C_0^{-1}\int\limits_{v(u)\leq -k,u\notin \varpi^{-k}(-1+\varpi O_F)}\mu_1(-\frac{\varpi^{k}}{1+u\varpi^{k}})\mu_2(-\alpha u)|\frac{\varpi^{k}}{\alpha u(1+u\varpi^{k})}|^{1/2}\psi(-\alpha u)p^{-v(\alpha)}du.
\end{equation}

\begin{defn}

We will say that a function $f(x)$  consists of level $i$ components, if
\begin{equation}
f(x)=\sum\limits_{\nu,n}a_{\nu,n}\charf_{\nu,n},
\end{equation}
where $a_{\nu,n}\neq 0$ only if $\nu$ is of level $i$.


By $L^2$ norm of a sequence of numbers $\{a_i\}$, we mean
\begin{equation}
(\sum |a_i|^2)^{1/2}.
\end{equation}

We will say that $f(x)$ consists of level $i$ components with coefficients of $L^2$ norm $h$,
if $f(x)$ consists of level $i$ components and the sequence of coefficients $\{a_{\nu,n}\}$ is of $L^2$ norm $h$.

\end{defn}


\begin{prop}\label{propWiofram}
Let $\pi$ be a supercuspidal representation of level $c$, or induced representation $\pi(\mu_1,\mu_2)$ where $\mu_1$ and $\mu_2$ are both ramified of level $k=c/2$. Let $W$ be the normalized Whittaker function  for a local new form of $\pi$, and $W^{(i)}$ be as in Definition \ref{defofWi}.
 \begin{enumerate}
\item $W^{(c)} (\alpha)=\charf_{1,0}(\alpha)$.
\item For $i=c-1>1$, $W^{(c-1)} (\alpha)$ is supported only at $v(\alpha)=0$, consisting of level $1$ components with coefficients of $L^2$ norm $\sqrt{\frac{p(p-2)}{(p-1)^2}}$, and also level 0 component with coefficient being $-\frac{1}{p-1}$.
\item In general for $0\leq i<c-1$, $i\neq k$, $W^{(i)} (\alpha)$ is supported only at $v(\alpha)=\min\{0, 2i-c\}$, consisting of level $c-i$ components with coefficients of $L^2$ norm $1$.
\item When $i=k>1$, $W^{(c/2)}$ is supported at $v(\alpha)\geq 0$, consisting of level $c/2$ components with coefficients of $L^2$ norm $1$. 

\noindent 
When $i=k=1$, $W^{(1)} (\alpha)$ consists of level 0 component at $v(\alpha)=0$ with coefficient being $-\frac{1}{p-1}$, and level $1$ components at $v(\alpha)\geq 0$ with coefficients of $L^2$ norm $\sqrt{\frac{p(p-2)}{(p-1)^2}}$.
\end{enumerate}
\end{prop}

\begin{rem}
In part (2) the coefficients for level 1 components together with level 0 components have $L^2$ norm $1$.

When $\pi$ is supercuspidal, one can actually get that $W^{(i)} (\alpha)$ consists of level $c-i$ components with coefficients of absolute value $\sqrt{\frac{p}{(p-1)^2p^{c-i-1}}}$. 
Counting the number of level $c-i$ characters, one can see that this is consistent with the $L^2$ norm as claimed.
This is however not necessarily true for induced representations.
\newline
\end{rem}

\begin{proof}
Let $\pi$ be a supercuspidal representation first.
 Then the only difference of this result from Corollary 2.18 in \cite{YH14} is the claim about the coefficients. Part (2)  can be easily proved using Lemma \ref{lemofGaussint} and (\ref{absolutevalueofC}).
Part (3) and (4) follow simply from the invariance of the unitary pairing (\ref{localnormalization2}).

Now let $\pi=\pi(\mu_1,\mu_2)$ where $\mu_1$, $\mu_2$ are both ramified of level $k=c/2$. 
The proof refines that of Lemma 4.2 in \cite{YH14} by using Lemma \ref{lemoflocalintofcharacters} above. For conciseness we will only prove part (4) when $i=k>1$. So
\begin{equation}
W^{(k)}(\alpha)=C_0^{-1}\int\limits_{v(u)\leq -k,u\notin \varpi^{-k}(-1+\varpi O_F)}\mu_1(-\frac{\varpi^{k}}{1+u\varpi^{k}})\mu_2(-\alpha u)\psi(-\alpha u)p^{-\frac{1}{2}v(\alpha)+v(u)}du.
\end{equation}
Let $\nu$ be an multiplicative unitary character. 
\begin{align}
&\text{\ \ \ \ }\int\limits_{v(\alpha)\text{\ fixed}}W^{(k)}(\alpha)\nu(\alpha)d^*\alpha\\
&=C_0^{-1}\int\limits_{v(u)\leq -k,u\notin \varpi^{-k}(-1+\varpi O_F)}(\int\limits_{v(\alpha)\text{\ fixed}} \nu(\alpha u)\mu_2(-\alpha u)\psi(-\alpha u)d^*\alpha)
\mu_1(-\frac{\varpi^{k}}{1+u\varpi^{k}})\nu^{-1}(u)p^{-\frac{1}{2}v(\alpha)+v(u)}du.\notag
\end{align}
For each fixed $v(u)$, the integral $\int\limits_{v(\alpha)\text{\ fixed}} \nu(\alpha u)\mu_2(-\alpha u)\psi(-\alpha u)d^*\alpha$ is actually independent of $u$ by a change of variable. Then by Lemma \ref{lemoflocalintofcharacters}, the integral in $u$ for fixed $v(u)$ is not zero only for $\nu$ of level $k$. 

Then as functions in $\alpha$, $\nu(\alpha u)\mu_2(-\alpha u)$ is of level $\leq k$, $\psi(-\alpha u)$ is of level $-v(\alpha)-v(u)\geq k-v(\alpha)$. Then $v(\alpha)\geq 0$ for the integral $\int\limits_{v(\alpha)\text{\ fixed}} \nu(\alpha u)\mu_2(-\alpha u)\psi(-\alpha u)d^*\alpha$ to be possibly nonzero. 

The claim about the $L^2$ norm of the coefficients follows from the invariance of the unitary pairing (\ref{localnormalization2}) as in the supercuspidal case.
\end{proof}
We can say something more for the case $i=c/2$. When $\pi$ is supercuspidal, $W^{(i)}$ is supported at $v(\alpha)\geq 0$, but it will also be bounded from above. This is however not true for induced representations.
\begin{lem}\label{lemofgrowthofWmidlevel}
Let $\pi=\pi(\mu_1,\mu_2)$ be an induced representation where both $\mu_i$ are ramified of level $k$. Then
\begin{equation}
W^{(k)}(a)<<_k p^{(\alpha-1/2)v(a)}v(a).
\end{equation}
Here $\alpha$ is a bound towards Ramanujan conjecture and we can pick $\alpha\leq 7/64$.
\end{lem}
\begin{proof}
Let $\nu$ be a character of the local field which is trivial on a fixed uniformizer. As in the proof of the last lemma,
\begin{align}
&\text{\ \ \ \ }\int\limits_{v(a)\text{\ fixed}}W^{(k)}(a)\nu(a)d^*a\\
&=C_0^{-1}\int\limits_{v(u)\leq -k,u\notin \varpi^{-k}(-1+\varpi O_F)}(\int\limits_{v(a)\text{\ fixed}} \nu(a u)\mu_2(-a u)\psi(-a u)d^*a)
\mu_1(-\frac{\varpi^{k}}{1+u\varpi^{k}})\nu^{-1}(u)p^{-\frac{1}{2}v(a)+v(u)}du.\notag
\end{align}
When $v(a)>2k$, we can further separate the integral above into two parts:
\begin{equation}\label{formulaasym2.1}
 I_1=\int\limits_{-2k\leq v(u)\leq -k,u\notin \varpi^{-k}(-1+\varpi O_F)}(\int\limits_{v(a)\text{\ fixed}} \nu(a u)\mu_2(-a u)d^*a)
\mu_1(-\frac{\varpi^{k}}{1+u\varpi^{k}})\nu^{-1}(u)p^{-\frac{1}{2}v(a)+v(u)}du,
\end{equation}
and
\begin{equation}\label{formulaasym2.2}
 I_2=\int\limits_{v(u)<-2k}(\int\limits_{v(a)\text{\ fixed}} \nu(a u)\mu_2(-a u)\psi(-a u)d^*a)
\mu_1(-\frac{1}{u})\nu^{-1}(u)p^{-\frac{1}{2}v(a)+v(u)}du.
\end{equation}
Note that the first integral will be non-zero only if $\nu$ is essentially $\mu_2^{-1}$. (This means they are identical on units, but can differ on a uniformizer.)
Then it's clear that for $-2k\leq v(u)\leq -k$,

\begin{equation}
 \int\limits_{v(a)\text{\ fixed}} \nu(a u)\mu_2(-a u)d^*a<<_kp^{\alpha v(a)}   \text{,\ \ \ } I_1<<_k p^{(\alpha-\frac{1}{2})v(a)}.
\end{equation}

The second integral is non-zero only if $\nu$ is essentially $\mu_1^{-1}$. Then we look at 
\begin{equation}
\int\limits_{v(a)\text{\ fixed}} \nu(a u)\mu_2(-a u)\psi(-a u)d^*a.
\end{equation}
If $\mu=\frac{\mu_1}{\mu_2}$ is unramified, then the non-zero contribution comes from $v(u)\geq -v(a)-1$, and
\begin{equation}
 I_2<<_kp^{(\alpha-1/2)v(a)}v(a).
\end{equation}
When $\mu$ is of level $j$, then the non-zero contribution comes from $v(u)=-v(a)-j$ (so that $\psi(a u)$ is of level $j$). Then 
\begin{equation}
 I_2<<_k p^{(\alpha-\frac{1}{2})v(a)}.
\end{equation}
Combining the bounds for $I_1$ $I_2$ and the restriction for $\nu$, the claim in the lemma is then clear.
\end{proof}

\section{The first inequality when testing on Cusp froms}
In this section we will prove Theorem \ref{thmoffirstineq} when $\phi$ is a Maass eigencuspform. In this case $\mu(\phi)=0$ directly. So we just need to prove the same inequality for $\mu_f(\phi)$.

The idea is to adelize $\phi$, $f$ and $f'=\bar{f}y^k$. Then $\mu_f(\phi)$ becomes an automorphic integral as in the triple product formula. Using (\ref{Globaltriple}), it would be enough to apply the weak subconvexity bound for the triple product L-function (see \cite{So10}), and give a reasonable upper bound for the normalized local integrals.

We will start with necessary tools to give upper bound for the local integrals. Section 3.1-3.3 will be purely local, so we will omit subscript $v$ without confusion.
\subsection{Matrix coefficient for highly ramified representations at non-archimedean places}
Locally let $\pi$ be a supercuspidal representation of level $c$ or of form $\pi(\mu_1,\mu_2)$ where $\mu_1$ and $\mu_2$ are both ramified of level $k=c/2$. Let $\varphi$ be a new form for $\pi$ which is invariant under $K_1(\varpi^{c})$
Let
\begin{equation}
\Phi(g)=<\pi(g)\varphi,\varphi>
\end{equation}
be the matrix coefficient associated to $\varphi$.
It is bi-$K_1(\varpi^{c})-$invariant. But we will only make use of the right $K_1(\varpi^{c})-$invariance now. 
By Lemma \ref{Iwasawadecomp}, to understand $\Phi(g)$, it will be enough to understand $\Phi(\zxz{a}{m}{0}{1}\zxz{1}{0}{\varpi^i}{1})$ for $0\leq i\leq c$.
Let $p=|\varpi|$. The following result is a refinement of Lemma 4.2 in \cite{YH14}.
\begin{prop}\label{propofsupportofMC}
Let $\Phi$ be as in the above notation.
\begin{enumerate}
\item[(i)] For $1<c-1\leq i\leq c$, $\Phi(\zxz{a}{m}{0}{1}\zxz{1}{0}{\varpi^i}{1})$ is supported on $v(a)=0$ and $v(m)\geq -1$. On the support, we have
\begin{equation}\label{matrixcoefforsceq1}
\Phi(\zxz{a}{m}{0}{1}\zxz{1}{0}{\varpi^i}{1})=\begin{cases}
1,&\text{\ if\ }v(m)\geq 0 \text{\ and\ }i=c;\\
-\frac{1}{p-1},&\text{\ if\ }v(m)=-1 \text{\ and\ }i=c;\\
-\frac{1}{p-1},&\text{\ if\ }v(m)\geq 0 \text{\ and\ }i=c-1.\\
\end{cases}
\end{equation}
When $v(a)=0$, $v(m)=-1$ and $i=c-1>1$, 
$\Phi(\zxz{a}{m}{0}{1}\zxz{1}{0}{\varpi^{c-1}}{1})$ consists of level 1 component with coefficients of $L^2$ norm $\frac{p\sqrt{p-2}}{(p-1)^2}$, and also level 0 component with coefficient $\frac{1}{(p-1)^2}$.
\item[(ii)] For $0\leq i<c-1$, $i\neq c/2$, $\Phi(\zxz{a}{m}{0}{1}\zxz{1}{0}{\varpi^i}{1})$ is supported on $v(a)=\min\{0,2i-c\}$, $v(m)=i-c$. As a function in $a$ it consists of level $c-i$ components with coefficients of $L^2$ norm $\sqrt{\frac{p}{(p-1)^2p^{c-i-1}}}$.
\item[(iii)]
When $c$ is even and $i=c/2>1$, $\Phi(\zxz{a}{m}{0}{1}\zxz{1}{0}{\varpi^i}{1})$ is supported on $v(a)\geq 0$, $v(m)=-c/2$. As a function in $a$ it consists of level $c-i$ components with coefficients of $L^2$ norm $\sqrt{\frac{p}{(p-1)^2p^{c/2-1}}}$.

\noindent
When $i=c/2=1$, $\Phi(\zxz{a}{m}{0}{1}\zxz{1}{0}{\varpi^i}{1})$ is supported on $v(a)\geq 0$, $v(m)\geq -1$. When $v(m)\geq 0$, its value is as in (i). When $v(m)=-1$, as a function in $a$ it consists of  level 0 component at $v(a)=0$ with coefficient $\frac{1}{(p-1)^2}$, and level $1$ components at $v(a)\geq 0$ with coefficients of $L^2$ norm $\frac{p\sqrt{p-2}}{(p-1)^2}$.
\end{enumerate}
\end{prop}
\begin{rem}
 The second part of (iii) looks like a combination of (i) and the first part of (iii). This case would not make essential difference for, for example, the bound of local integral of triple product formula in Proposition \ref{propoflocalboundI} and \ref{propoflocalboundII}.
\end{rem}

\begin{proof}
By definition,
\begin{align}\label{formulaofMC}
\Phi(\zxz{a}{m}{0}{1}\zxz{1}{0}{\varpi^i}{1})&=\int_{\F_v^*}\pi(\zxz{a}{m}{0}{1}\zxz{1}{0}{\varpi^i}{1}) W(\zxz{\alpha}{0}{0}{1})\overline{W(\zxz{\alpha}{0}{0}{1})}d^*\alpha\\
&=\int\limits_{\F_v^*}\psi(m\alpha)W^{(i)}(a\alpha)\overline{W^{(c)}}(\alpha)d^*\alpha\notag\\
&=\int\limits_{v(\alpha)=0}\psi(m\alpha)W^{(i)}(a\alpha)d^*\alpha\notag
\end{align}

To get a non-zero value for $\Phi$, we just need a level 0 component supported at $v(\alpha)=0$ for $\psi(m\alpha)W^{(i)}(a\alpha)$. Then the claims follow from Proposition \ref{propWiofram}. For conciseness we will only prove part (ii) here. 

Let $i<c-1$, $i\neq c/2$. According to part (3) of Proposition \ref{propWiofram}, $W^{(i)}(x)$ is supported at $v(x)=\min\{0, 2i-c\}$. So (\ref{formulaofMC}) is not zero only if $v(a)=\min\{0, 2i-c\}$. We further know that $W^{(i)}(a\alpha)$ consists only of level $c-i$ characters in $\alpha$. Then to get level 0 component for the product $\psi(m\alpha)W^{(i)}(a\alpha)$ at $v(\alpha)=0$, we need $v(m)=i-c$. 

Now suppose $i>c/2$ so that $W^{(i)}(x)$ is supported at $v(x)=0$. (The case when $i<c/2$ is very similar.)
If we write
\begin{equation}
W^{(i)}(x)=\sum\limits_{\nu \text{\ of level \ }c-i}c_\nu \nu(x)
\end{equation}
for $\alpha\in O_v^*$. Then
\begin{equation}
\Phi(\zxz{a}{m}{0}{1}\zxz{1}{0}{\varpi^i}{1})=\int\limits_{v(\alpha)=0}\psi(m\alpha)W^{(i)}(a\alpha)d^*\alpha=\sum\limits_{\nu \text{\ of level \ }c-i}c_\nu (\int\limits_{v(\alpha)=0}\psi(m\alpha)\nu(\alpha)d^*\alpha) \nu(a).
\end{equation}
So as a function in $a$, $\Phi(\zxz{a}{m}{0}{1}\zxz{1}{0}{\varpi^i}{1})$ consists of level $c-i$ components with coefficients of $L^2$ norm $\sqrt{\frac{p}{(p-1)^2p^{c-i-1}}}$. This is because the sequence $\{c_\nu\}$ is of $L^2$ norm $1$ and 
\begin{equation*}
|\int\limits_{v(\alpha)=0}\psi(m\alpha)\nu(\alpha)d^*\alpha|=\sqrt{\frac{p}{(p-1)^2p^{c-i-1}}}
\end{equation*}
for all $\nu$ of level $c-i$ at $v(m)=i-c$.

One can prove the other parts similarly. In particular (i) follows from (1)(2) of Proposition \ref{propWiofram}, and (iii) follows from (4) of Proposition \ref{propWiofram}.
\end{proof}

\begin{cor}\label{corMCofoldform}
 Let $\widetilde{\Phi}$ be the matrix coefficient associated to $\pi(\zxz{\varpi^{-n}}{0}{0}{1})\varphi$, where $\varphi$ is still a new form. Then $\widetilde{\Phi}$ is right $K_1(\varpi^{c+n})-$invariant, and
 $\widetilde{\Phi}(\zxz{a}{m}{0}{1}\zxz{1}{0}{\varpi^i}{1})$ is supported at $v(m)\geq -n-1$ for $i=c+n \text{\ or\ }c+n-1$, and $v(m)=i-2n-c$ for $i<c+n-1$.
\end{cor}
\begin{proof}
 Let $\Phi$ be the matrix coefficient associated to the new form as in Proposition \ref{propofsupportofMC}. Then
 \begin{equation}
  \widetilde{\Phi}(\zxz{a}{m}{0}{1}\zxz{1}{0}{\varpi^i}{1})=\Phi(\zxz{a}{m\varpi^n}{0}{1}\zxz{1}{0}{\varpi^{i-n}}{1}).
 \end{equation}
When $i\geq n$, we can use Proposition \ref{propofsupportofMC} directly to get $\widetilde{\Phi}$. When $i<n$, we have
\begin{equation}
 \zxz{a}{m\varpi^n}{0}{1}\zxz{1}{0}{\varpi^{i-n}}{1}=\varpi^{i-n}\zxz{a\varpi^{-2i+2n}}{a(\varpi^{-i+n}-\varpi^{-2i+2n})+m\varpi^{n}}{0}{1}\zxz{1}{0}{1}{1}\zxz{1}{-1+\varpi^{-i+n}}{0}{1}.
\end{equation}
Then by Proposition \ref{propofsupportofMC}, $\widetilde{\Phi}$ is nonzero only when $v(a\varpi^{-2i+2n})=-c$ and $v(a(\varpi^{-i+n}-\varpi^{-2i+2n})+m\varpi^{n})=-c$. Note that $v(a(\varpi^{-i+n}-\varpi^{-2i+2n}))=-c+i-n<-c$ in this case, which forces $v(m)=i-2n-c$.
\end{proof}

\begin{rem}
 Using the same proof, one can get more detailed descriptions of matrix coefficient of old forms as in Proposition \ref{propofsupportofMC}. But we won't go into details here.
\end{rem}

\subsection{Bound of local triple product integral I}
\begin{prop}\label{propoflocalboundI}
Let $\pi_1$ be highly ramified of level $c_1>1$ and $\pi_2$ $\pi_3$ be highly ramified of level $c>c_1$. Suppose that they all have trivial central characters. Let $\Phi_i$ be the normalized matrix coefficients associated to the new forms of $\pi_i$. 
Then
\begin{equation}
 |I_v|=|\int\limits_{\Q_v^*\backslash \GL_2^*(\Q_v)}\prod\limits_{j=1}^{3}\Phi_j(g)dg|\leq \frac{2}{(p+1)p^{c-1}}\frac{p^3-2p^2+1}{(p-1)^3}\leq 4p^{-c}.
\end{equation}
\end{prop}
\begin{proof}
By Lemma \ref{Iwasawadecomp} and \ref{localintcoefficient}, we can decompose the integral as
\begin{equation}\label{triplelocalintI}
\sum\limits_{i=0}^{c}A_i\int\limits_{a,m}\prod\limits_{j=1}^{3}\Phi_j(\zxz{a}{m}{0}{1}\zxz{1}{0}{\varpi^i}{1})|a|^{-1}d^*a dm.
\end{equation}
By our assumption on $\pi_i$, Proposition \ref{propofsupportofMC} holds for $\Phi_i$. In particular when $i<c-1$, the support of $\Phi_1$ is disjoint from the support of $\Phi_2$ or $\Phi_3$. So one only need to consider $i=c$ or $c-1$ in (\ref{triplelocalintI}). When $i=c$, the corresponding integral is
\begin{equation}
\int\limits_{v(a)=0,v(m)\geq -1}\Phi_1(\zxz{a}{m}{0}{1})\Phi_2(\zxz{a}{m}{0}{1})\Phi_3(\zxz{a}{m}{0}{1})d^*a dm=1+(p-1)(-\frac{1}{(p-1)^3})=\frac{p^2-2p}{(p-1)^2}.
\end{equation}
When $i=c-1$ and $v(m)\geq 0$,
\begin{align}
&\text{\ \ \ }\int\limits_{v(a)=0,v(m)\geq 0}\Phi_1(\zxz{a}{m}{0}{1})\Phi_2(\zxz{a}{m}{0}{1}\zxz{1}{0}{\varpi^{c-1}}{1})\Phi_3(\zxz{a}{m}{0}{1}\zxz{1}{0}{\varpi^{c-1}}{1})d^*a dm=\frac{1}{(p-1)^2}. 
\end{align}
When $i=c-1$ and $v(m)=-1$,
\begin{align}\label{formula1}
&\text{\ \ \ }\int\limits_{v(a)=0,v(m)=-1}\Phi_1(\zxz{a}{m}{0}{1})\Phi_2(\zxz{a}{m}{0}{1}\zxz{1}{0}{\varpi^{c-1}}{1})\Phi_3(\zxz{a}{m}{0}{1}\zxz{1}{0}{\varpi^{c-1}}{1})d^*a dm\\
&=\int\limits_{v(a)=0,v(m)=-1}-\frac{1}{p-1}\Phi_2(\zxz{a}{m}{0}{1}\zxz{1}{0}{\varpi^{c-1}}{1})\Phi_3(\zxz{a}{m}{0}{1}\zxz{1}{0}{\varpi^{c-1}}{1})d^*a dm.\notag
\end{align}
When $i=c-1$, $v(m)=-1$ and $j=2,3$, $\Phi_j(\zxz{a}{m}{0}{1}\zxz{1}{0}{\varpi^{c-1}}{1})$ as functions in $a$ consist of level 1 components with coefficients of $L^2$ norm $\frac{p\sqrt{p-2}}{(p-1)^2}$, and also level 0 component with coefficient $\frac{1}{(p-1)^2}$. For fixed $m$, suppose we can write
\begin{equation}
\Phi_2(\zxz{a}{m}{0}{1}\zxz{1}{0}{\varpi^{c-1}}{1})=\sum\limits_{\nu \text{\ of level 1}} a_\nu\nu(a)+\frac{1}{(p-1)^2},
\end{equation}
and
\begin{equation}
\Phi_3(\zxz{a}{m}{0}{1}\zxz{1}{0}{\varpi^{c-1}}{1})=\sum\limits_{\nu \text{\ of level 1}} b_\nu\nu(a)+\frac{1}{(p-1)^2}.
\end{equation}
Then the coefficient of the level 0 component of the product $\Phi_2\Phi_3$ is 
\begin{equation}
\sum\limits_{\nu \text{\ of level 1}}a_\nu b_{\nu^{-1}}+\frac{1}{(p-1)^4}\leq\frac{p^2(p-2)}{(p-1)^4}+\frac{1}{(p-1)^4}=\frac{p^3-2p^2+1}{(p-1)^4}.
\end{equation}
Here we have used Cauchy-Schwartz inequality.
Then if we integrate in $a$ first for (\ref{formula1}), we can get
\begin{equation}
|\int\limits_{v(a)=0,v(m)=-1}-\frac{1}{p-1}\Phi_2(\zxz{a}{m}{0}{1}\zxz{1}{0}{\varpi^{c-1}}{1})\Phi_3(\zxz{a}{m}{0}{1}\zxz{1}{0}{\varpi^{c-1}}{1})d^*a dm|\leq \frac{p^3-2p^2+1}{(p-1)^4}.
\end{equation}
Now put all pieces together, we have
\begin{equation}
|I_v|\leq \frac{1}{(p+1)p^{c-1}}\frac{p^2-2p}{(p-1)^2}+\frac{p-1}{(p+1)p^{c-1}}[\frac{1}{(p-1)^2}+\frac{p^3-2p^2+1}{(p-1)^4}]\leq \frac{2}{(p+1)p^{c-1}}\frac{p^3-2p^2+1}{(p-1)^3}.
\end{equation}
Lastly 
\begin{equation}
 \frac{2}{(p+1)p^{c-1}}\frac{p^3-2p^2+1}{(p-1)^3}\leq 4p^{-c}
\end{equation}
as $p\geq 2$.
\end{proof}

\begin{prop}\label{propoflocalboundIold}
 Let $\pi_i,i=1,2,3$ and $\Phi_2$ $\Phi_3$ be as in Proposition \ref{propoflocalboundI}. Let $\Phi_1$ be the matrix coefficient associated to an old form $\pi_1(\zxz{\varpi^{-n}}{0}{0}{1})f_1$, where $f_1$ is still a new form. 
 Suppose that $c_1+2n<c$. Then 
 \begin{equation}
  |I_v|=|\int\limits_{\Q_v^*\backslash \GL_2^*(\Q_v)}\prod\limits_{j=1}^{3}\Phi_j(g)dg|\leq  \frac{4}{p^c}p^n.
 \end{equation}

\end{prop}
\begin{proof}
 In this case we apply Corollary \ref{corMCofoldform} to $\Phi_1$. By the assumption that $c_1+2n<c$, the support of $\Phi_1$ in $m$ will still be disjoint from the support of $\Phi_2$ and $\Phi_3$ when $i<c-n-1$. When $i\geq c-n-1$, $\Phi_1$ will be of level 0 in $a$.
 Then one  can look at the level 0 components of $\Phi_2\Phi_3$ on the support and use Cauchy-Schwartz inequality as in the proof of previous proposition. 
 \end{proof}
\begin{rem}
The bound is a little loose but will serve the purpose. 
\end{rem}

\subsection{Bound of local triple product integral II}
In this subsection we consider the case when $\pi_1$ is an unramified special representation, that is, an unramified twist of steinberg representation.
\begin{prop}\label{propoflocalboundII}
Let $\pi_1$ be an unramified special representation and $\pi_2$ $\pi_3$ be highly ramified of level $c>1$. Suppose that they all have trivial central characters. Let $\Phi_i$ be the normalized matrix coefficients associated to the new forms of $\pi_i$. 
Then
\begin{equation}
 |I_v|=|\int\limits_{\Q_v^*\backslash \GL_2^*(\Q_v)}\prod\limits_{j=1}^{3}\Phi_j(g)dg|\leq 4p^{-c}.
\end{equation}
\end{prop}
\begin{proof}

Before we start, we first recall the matrix coefficient for unramified special representation from \cite{MW12}.

Let $\sigma_n=\zxz{\varpi^n}{0}{0}{1}$, and $\omega=\zxz{0}{1}{-1}{0}$. 
\begin{lem}\label{lemofMCforSpecial}
 Let $\pi=\sigma(\chi|\cdot|^{1/2},\chi|\cdot|^{-1/2})$ be an unramified special  representation of $\GL_2$ with $\chi$ unramified unitary. It has a normalized $K_1(\varpi)-$invariant new form. The associated matrix coefficient $\Phi$ for this new form is bi-$K_1(\varpi)-$invariant and can be given in the following table for double $K_1(\varpi)-$cosets:
 
 \begin{tabular}{|p{2cm}|p{1.8cm}|p{1.8cm}|p{1.8cm}|p{1.8cm}|p{1.8cm}|p{1.8cm}|}
  \hline
  g	&$1$	&$\omega$	&$\sigma_n$	&$\omega\sigma_n$	&$\sigma_n\omega$	&$\omega\sigma_n\omega$\\\hline
  $\Phi(g)$	&$1$	&$-p^{-1}$	&$\chi^np^{-n}$	&$-\chi^np^{1-n}$ &$-\chi^np^{-1-n}$	&$\chi^np^{-n}$	\\\hline	
 \end{tabular}
 
 In this table $n\geq 1$.
\end{lem}

As in the last subsection, we can split the integral $I_v$ as 
\begin{equation}\label{triplelocalintII}
I_v=\sum\limits_{i=0}^{c}A_i\int\limits_{a,m}\prod\limits_{j=1}^{3}\Phi_j(\zxz{a}{m}{0}{1}\zxz{1}{0}{\varpi^i}{1})|a|^{-1}d^*a dm.
\end{equation}
We can use the support of $\Phi_2$ and $\Phi_3$ to simplify the calculations. In particular we put all the information necessary for the integral into the following table.

\begin{tabular}{|p{3cm}|p{1.5cm}|p{2.5cm}|p{5cm}|p{3cm}|}
 \hline
 Cases	&$A_i$	&$\Phi_1$	& coefficient of level 0 component of $\Phi_2\Phi_3$	&$|a|^{-1}d^*adm$\\\hline
 $i=c$, $v(a)=0$, $v(m)\geq 0$	&$\frac{1}{(p+1)p^{c-1}}$ &$1$	&$1$	&$1$\\\hline
 $i=c$, $v(a)=0$, $v(m)=-1$ &$\frac{1}{(p+1)p^{c-1}}$	&$-p^{-1}$	&$\frac{1}{(p-1)^2}$	&$p-1$\\\hline
 $i=c-1$, $v(a)=0$, $v(m)\geq 0$	&$\frac{p-1}{(p+1)p^{c-1}}$	&$1$	&$\frac{1}{(p-1)^2}$	&$1$\\\hline
 $i=c-1$, $v(a)=0$, $v(m)=-1$	&$\frac{p-1}{(p+1)p^{c-1}}$	&$-p^{-1}$	&bdd by $\frac{p^3-2p^2+1}{(p-1)^4}$	&$p-1$\\\hline
$c/2<i<c-1$, $v(a)=0$, $v(m)=i-c$	&$p^{-i}\frac{p-1}{(p+1)}$	&$-p^{1+2i-2c}$	&bdd by $\frac{p}{(p-1)^2p^{c-1-i}}$	&$(p-1)p^{c-i-1}$\\\hline
 $0<i<c/2$, $v(a)=2i-c$, $v(m)=i-c$	&$p^{-i}\frac{p-1}{(p+1)}$	&$-\chi^{2i-c}p^{1-c}$	&bdd by $\frac{p}{(p-1)^2p^{c-1-i}}$	&$p^{2i-c}(p-1)p^{c-i-1}$\\\hline
 $i=0$, $v(a)=-c$, $v(m)=-c$	&$\frac{p}{(p+1)}$	&$
                                              -\chi^{-c}p^{1-c}$,  if
                                              $v(a+m)>-c$;
                                              $\chi^{-c}p^{-c}$,  if
                                              $v(a+m)=-c$.
                                             	&coefficients of level 1 and level 0 components bdd by $\frac{p}{(p-1)^2p^{c-1}}$	&$p^{-c}(p-1)p^{c-1}$\\\hline

 \end{tabular}

The column for $\Phi_1$ is a reformulation of the results of Lemma \ref{lemofMCforSpecial} in terms of double $B-K_1(\varpi)$ cosets on the support of $\Phi_2$.
For the column of $\Phi_2\Phi_3$ we have used Proposition \ref{propofsupportofMC}, and Cauchy-Schwartz inequality whenever we know only $L^2$ norm of the coefficients for $\Phi_2$ and $\Phi_3$.
One only has to care about the level 0 component of $\Phi_2\Phi_3$ when $i\neq 0$  because the value of $\Phi_1$ is constant on each support.

The first observation is that when $0<i<c-1$ becomes smaller, the contribution from the corresponding integral is bounded by 
$$\frac{p^{2-2c}}{p+1}p^{i},$$
which is a geometric sequence whose sum can be nicely bounded. Note that when we move on to $0<i<c/2$ from $c/2<i<c-1$, the ``discontinuity'' comes from $\Phi_1$ and the Haar measure $|a|^{-1}d^*adm$. But their change just cancel each other when we take absolute values. (Note that $|\chi|=1$.)
We didn't put $i=c/2$ in the table, as Proposition \ref{propofsupportofMC} only claims $v(a)\geq 0$ in that case. But whenever there is a level 0 component of $\Phi_2\Phi_3$ supported at $v(a)>0$, one can check that $|\Phi_1|=|-\chi^{v(a)}p^{1-c}p^{-v(a)}|=p^{1-c}p^{-v(a)}$ and there will be an additional factor $p^{v(a)}$ coming from the Haar measure $|a|^{-1}d^*adm$. So the effect of $v(a)>0$ will be cancelled, and
the upper bound $$\frac{p^{2-2c}}{p+1}p^{i}$$ is true also for $i=c/2$ case.

The case when $i=0$ is slightly more complicated, as the level 1 components of $\Phi_2\Phi_3$ will also contribute to the final integral.  When $i=0$, $v(a)=v(m)=-c$,
\begin{equation}
 \Phi_1(\zxz{a}{m}{0}{1}\zxz{1}{0}{1}{1})=\begin{cases}
                                             -\chi^{-c}p^{1-c}, &\text{\ if\ }v(a+m)>-c;\\
                                            \chi^{-c}p^{-c}, &\text{\ if\ }v(a+m)=-c.
                                           \end{cases}
\end{equation}
As a result of this,
\begin{equation}
 |\int\limits_{v(a)=-c} \Phi_1(\zxz{a}{m}{0}{1}\zxz{1}{0}{1}{1})d^*a|\leq p^{1-c}\frac{2}{p-1},
\end{equation}
and
\begin{equation}
 |\int\limits_{v(a)=-c} \Phi_1(\zxz{a}{m}{0}{1}\zxz{1}{0}{1}{1})\nu(a)d^*a|\leq p^{1-c}(1+p^{-1})
\end{equation}
for all character $\nu$ of level 1.

Then the contribution from $i=0$ is bounded by 
\begin{equation}
 \frac{p}{p+1}p^{1-c}[\frac{2}{p-1}+(p-2)(1+p^{-1})]\frac{p}{(p-2)^2p^{c-1}}(1-p^{-1})=\frac{p^{3-2c}}{p^2-1}\frac{p^3-2p^2+p+2}{p(p-1)}.
\end{equation}
Combining all the pieces, we can get
\begin{align}
 |I_v|&\leq \frac{1}{(p+1)p^{c-1}}[1-\frac{1}{p(p-1)}]+\frac{p-1}{(p+1)p^{c-1}}[\frac{1}{(p-1)^2}+p^{-1}\frac{p^3-2p^2+1}{(p-1)^3}]\\
 &\text{\ \ }+\sum\limits_{i=1}^{c-2}\frac{p^{2-2c}}{p+1}p^{i}+\frac{p^{3-2c}}{p^2-1}\frac{p^3-2p^2+p+2}{p(p-1)}\notag\\
 &\leq 4p^{-c}.\notag
\end{align}
We have used that $c\geq 2$ and $p\geq 2$ in the last inequality.
\end{proof}

\begin{prop}\label{propoflocalboundIIold}
  Let $\pi_i,i=1,2,3$ and $\Phi_2$ $\Phi_3$ be as in Proposition \ref{propoflocalboundII}. Let $\Phi_1$ be the matrix coefficient associated to an old form $\pi_1(\zxz{\varpi^{-n}}{0}{0}{1})f_1$, where $f_1$ is still a new form of $\pi_1$. 
 Suppose that $1+2n<c$. Then 
 \begin{equation}
  |I_v|=|\int\limits_{\Q_v^*\backslash \GL_2^*(\Q_v)}\prod\limits_{j=1}^{3}\Phi_j(g)dg|\leq  \frac{4}{p^c}p^{2n}.
 \end{equation}

\end{prop}
\begin{proof}
 As in Corollary \ref{corMCofoldform}, one can describe the matrix coefficient of an old form using the matrix coefficient of the new form in Lemma \ref{lemofMCforSpecial}. Then one can bound the local integral as in Proposition \ref{propoflocalboundII}. 
When $i\leq n$, one need to consider the components of $\Phi_2\Phi_3$ of level up to $1+n-i$. We shall skip the technicalities here. Again the bound we get here is pretty loose, but it will serve the purpose.
 \end{proof}

\subsection{Proof of Theorem \ref{thmoffirstineq} when $\phi$ is Maass eigencuspform}
For the cases considered in Proposition \ref{propoflocalboundI} and \ref{propoflocalboundII}, we can also get a bound for the normalized local integrals
\begin{equation}
I_v^0\leq 10^5p^{-c}.
\end{equation}
We are not very careful on bounding the normalizing L-factors, as any fixed constant multiple will eventually be bounded by $(q/\sqrt{C})^{\epsilon}$ for any $\epsilon>0$.
Note this bound is actually better than the bound obtained in Corollary 2.8 of \cite{PPA14}. (Of course that is for different cases.) 

When $\phi$ is an old form, we need to use Proposition \ref{propoflocalboundIold} and Proposition \ref{propoflocalboundIIold}, and there are extra $p^n$ or $p^{2n}$ factors in the local upper bounds. But when we take a product, such extra factors will be bounded by $N^2$, where $N$ is the fixed level of $\phi$. 
So they can be ignored when we discuss the asymptotic behavior of $f$.
Note that we haven't consider the case when the local component of $\phi$ is an old form from an unramified representation. For this case one can 
change the local triple product integral into Rankin-Selberg integral as in \cite{PPA14}, and we will give an upper bound for local Rankin-Selberg integral for old forms very explicitly in next section. So we will not consider this case in detail here.

From this point on we can use the same argument as in \cite{PPA14} to prove Theorem \ref{thmoffirstineq}. So we will be very brief. 
Suppose that after adelization, $f$ belongs to an automorphic cuspidal representation $\pi$. $f'$ also belongs to $\pi$, while its component at infinity is of weight $-k$. Similarly suppose that the adelization of $\phi$ belongs to $\pi_\phi$.

Recall that $C$ is the finite conductor of $\pi\times\pi$. The conductor of $\pi\times\pi\times\pi_\phi$ is then $\asymp_\phi C^2k^4$. The argument of \cite{So10} implies that 
\begin{equation}
L(\pi_\phi\times\pi\times\pi,1/2)<<\frac{\sqrt{C}k}{log(Ck)^{1-\epsilon}}.
\end{equation}
Then combining the local bounds (including Corollary 2.8 of \cite{PPA14}) with this weak subconvexity bound into the Triple product formula will prove the claim in Theorem \ref{thmoffirstineq}.

\section{Proof of Theorem \ref{thmoffirstineq} when testing on incomplete Eisenstein series}
In this section, we shall prove Theorem \ref{thmoffirstineq} when $\phi=E_\mathfrak{a}(z,\Psi)$ is an incomplete Eisenstein series of level $N$. Here $\mathfrak{a}$ is a cusp for $\Gamma_0(N)$ and $\Psi$ is a compactly supported function on $\R^+$.


According to (\ref{spectrumdecompofincompleteEis}), 
\begin{equation}
E_\mathfrak{a}(z,\Psi)=\frac{\hat{\Psi}(1)}{Vol(Y_0(N))}+
\frac{1}{2\pi i}\int\limits_{(1/2)}\hat{\Psi}(s)E_\mathfrak{a}(z,s)ds.
\end{equation}
So
\begin{equation}
\frac{\mu_f(E_\mathfrak{a}(z,\Psi))}{\mu_f(1)}=\frac{\hat{\Psi}(1)}{Vol(Y_0(N))}+\frac{1}{2\pi i\mu_f(1)}\int\limits_{(1/2)}\hat{\Psi}(s)\mu_f( E_\mathfrak{a}(z,s))ds.
\end{equation}
On the other hand,
\begin{align}
\mu(E_\mathfrak{a}(z,\Psi))&=\int\limits_{\Gamma_0(N)\backslash\H}\sum\limits_{\gamma\in \Gamma_0(N)_\mathfrak{a}\backslash\Gamma_0(N)}\Psi(\Im (\sigma_{\mathfrak{a}}^{-1}\gamma z)) \frac{dxdy}{y^2}
=\int\limits_{\Gamma_0(N)_\mathfrak{a}\backslash\H}\Psi(\Im (\sigma_{\mathfrak{a}}^{-1} z)) \frac{dxdy}{y^2}\\
&=\int\limits_{\Gamma_\infty\backslash\H}\Psi(\Im (z)) \frac{dxdy}{y^2}\notag
\end{align}
In the last equality, we have made a change of variable $\sigma_{\mathfrak{a}}^{-1} z\rightarrow z$, and that 
\begin{equation}
 \sigma_\mathfrak{a}^{-1}\Gamma_0(N)_\mathfrak{a}\sigma_\mathfrak{a}=\Gamma_\infty.
\end{equation}
Then
\begin{equation}\label{formula4.1}
\frac{\mu(E_\mathfrak{a}(z,\Psi))}{\mu(1)}=\frac{1}{\mu(1)}\int\limits_{x=-1/2}^{1/2}\int\limits_{y=0}^{\infty}\Psi(y)\frac{dxdy}{y^2} =\frac{\hat{\Psi}(1)}{Vol(Y_0(N))}.
\end{equation}
So to prove Theorem \ref{thmoffirstineq} for incomplete Eisenstein series, it would be enough to show the same inequality for 
\begin{equation}
\mu_f( E_\mathfrak{a}(z,s))
\end{equation}
 with uniform implied constant for all $s$ on the line $\Re(s)=1/2$.

The idea is then similar as in the case of Maass eigencuspforms. We shall adelize $f$ $f'$ and $E_\mathfrak{a}(z,s)$. Then $\mu_f(E_\mathfrak{a}(z,s))$ will be essentially the Rankin-Selberg integral. The result will follow from the weak subconvexity bound for the Rankin-Selberg L-function and a reasonable upper bound for the local Rankin-Selberg integral.
\subsection{Adelization of $E_\mathfrak{a}(z,s)$}
It is impossible however, to adelize $E_\mathfrak{a}(z,s)$ to be an automorphic Eisenstein series coming from a single induced representation. We shall first write $E_\mathfrak{a}(z,s)$ as a linear combination of related Eisenstein series, which in turn can be adelized to be simple automorphic Eisenstein series.
We believe that the contents of this subsection should be well-known to experts. But as we couldn't find a proper reference, we will work them out with a little more details.

Suppose that $\tau_\mathfrak{a}=\zxz{a}{b}{c}{d}$ where $c|N$ and $d\in (\Z/(c,N/c))^*$, as in Section \ref{subsecofcusps}. For simplicity, let $c'=(c,N/c)$. We also have the width of the cusp $d_\mathfrak{a}=\frac{N}{c c'}$.

First of all, note that
\begin{equation}
E_\mathfrak{a}(z,s)=d_\mathfrak{a}^{-s}\sum\limits_{\gamma\in \tau_\mathfrak{a}^{-1}\Gamma_0(N)_\mathfrak{a}\tau_\mathfrak{a}\backslash\tau_\mathfrak{a}^{-1}\Gamma_0(N)}\Im(\gamma z)^s,
\end{equation}
where $\tau_\mathfrak{a}^{-1}\Gamma_0(N)_\mathfrak{a}\tau_\mathfrak{a}$ is generated by $\pm 1$ and $\zxz{1}{d_\mathfrak{a}}{0}{1}$.
Let $\iota_\mathfrak{a}$ denote the injection
\begin{align}
 \iota_\mathfrak{a}:\tau_\mathfrak{a}^{-1}\Gamma_0(N)_\mathfrak{a}\tau_\mathfrak{a}\backslash\tau_\mathfrak{a}^{-1}\Gamma_0(N)&\hookrightarrow \Gamma_\infty\backslash \SL_2(\Z).\\
 \gamma\text{\ \ \ \ }&\mapsto \text{\ \ }\gamma\notag
 \end{align}
One can easily check by definition that this is indeed an injection, and 
\begin{equation}
 \Image(\iota_\mathfrak{a})=\Gamma_\infty\backslash \Gamma_\infty\tau_\mathfrak{a}^{-1}\Gamma_0(N).
\end{equation}
Since $\tau_\mathfrak{a}$'s are double coset representatives for $\Gamma_0(N)\backslash\SL_2(\Z)/\Gamma_\infty$, then $\Image(\iota_\mathfrak{a})$ are disjoint for different $\mathfrak{a}$'s, and
\begin{equation}
 \coprod\limits_{\mathfrak{a}}\Image(\iota_\mathfrak{a})=\Gamma_\infty\backslash \SL_2(\Z).
\end{equation}

Define
\begin{equation}
\Gamma(N,c)=\{\zxz{k_1}{k_2}{k_3}{k_4}\in \SL_2(\Z)| (k_3,N)=c\}.
\end{equation}
It is left $\Gamma_\infty-$invariant and right $\Gamma_0(N)-$ invariant, and
\begin{equation}
\SL_2(\Z)=\coprod\limits_{c|N}\Gamma(N,c).
\end{equation}
Recall that $\mathcal{C}[c]$ is the set of cusps whose denominator is the given number $c$.
Then we have
\begin{equation}
\coprod\limits_{\mathfrak{a}\in \mathcal{C}[c]}\text{\Image}(\iota_\mathfrak{a})=\Gamma_\infty\backslash \Gamma(N,c).
\end{equation}

\begin{defn}
 For fixed $N$, $c$ and $c'$, let $\chi$ be a Dirichlet character of level dividing $c'$. Define the following Eisenstein series associated to $\chi$:
\begin{equation}
E_{\chi,c}(z,s)=\sum\limits_{\gamma\in \Gamma_\infty\backslash\Gamma(N,c)}\Im(\gamma z)^s\chi(\frac{k_1c}{k_3}), \text{\ if\ }\gamma=\zxz{k_1}{k_2}{k_3}{k_4}.
\end{equation}
\end{defn}

Note that if $\gamma\in \Image(\iota_\mathfrak{a})$, then 
\begin{equation}
\zxz{k_1}{k_2}{k_3}{k_4}=\gamma'\zxz{d}{-b}{-c}{a}\gamma''
\end{equation}
for some $\gamma'\in\Gamma_\infty$ and $\gamma''\in \Gamma_0(N)$ and $\tau_\mathfrak{a}=\zxz{a}{b}{c}{d}$. Then
\begin{equation}
\chi(\frac{k_1c}{k_3})=\chi(-d)
\end{equation}
is constant on $\Image(\iota_\mathfrak{a})$. This implies the following
\begin{lem}
$(\frac{N}{c c'})^{s}E_\mathfrak{a}(z,s)$ is a linear combination of $E_{\chi,c}(z,s)$ for all Dirichlet characters $\chi$ of level dividing $c'$.
\end{lem}
The adelization of $E_{\chi,c}(z,s)$ is now easy to describe. Any Dirichlet character $\chi$ of level dividing $c'$ has an idelic lift of the same level, which we will also denote by $\chi$. 
Suppose
\begin{equation}
 c=\prod\limits_{p|c}p^{c_p},
\end{equation}
and 
\begin{equation}
 N=\prod\limits_{p|N}p^{e_p}.
\end{equation}

Let $\Phi_{s,c}$ be the element of $\text{Ind}_B^{\GL_2}(\chi|\cdot|^{s-1/2},\chi^{-1}|\cdot|^{-s+1/2})$ whose local component at $p$ is supported on 
\begin{equation}
 B\zxz{1}{0}{p^{c_p}}{1}K_0(p^{e_p}),
\end{equation}
and take value $1$ at $\zxz{1}{0}{p^{c_p}}{1}$. Then the adelization of $E_{\chi,c}(z,s)$ is
\begin{equation}
 E_{\chi,c}(g,s)=\sum\limits_{\gamma\in B(\Q)\backslash\GL_2(\Q)}\Phi_{s,c}(\gamma g).
\end{equation}
Recall that $N_\Diamond$ is the largest integer such that $N_\Diamond^2|N$.
\begin{defn}\label{defnofEisofleveluptoN}
 Define  $V_{E}(N)$ to be space spanned by $\Phi_s$,
where $\Phi_s$ runs over new forms or old forms of level dividing $N$ in $\text{Ind}_B^{\GL_2}(\chi|\cdot|^{s-1/2},\chi^{-1}|\cdot|^{-s+1/2})$, and $\chi$ runs over idelic lifts of Dirichlet characters of level dividing $N_\Diamond$.
Define the operator $Eis$ to be
 \begin{equation}
 Eis: \Phi_s \mapsto \sum\limits_{\gamma\in B(\Q)\backslash\GL_2(\Q)}\Phi_{s}(\gamma g).
 \end{equation}
\end{defn}

\begin{cor}\label{corofadelization}
 The adelization of $(\frac{N}{c c'})^{s}E_\mathfrak{a}(z,s)$ belongs to $Eis(V_{E}(N))$.
\end{cor}
The way we are going to use this result is as follows. The factor $(\frac{N}{c c'})^{s}$ is easily controlled on $\Re(s)=1/2$. To prove certain asymptotic property for $E_\mathfrak{a}(z,s)$, it would then be enough to prove the same asymptotic property for a basis of $Eis(V_{E}(N))$. For different purposes, we might choose slightly different basis.

\subsection{Bounding local integral for Rankin-Selberg integral}
In this subsection we will give a reasonable upper bound for the following local  Rankin-Selberg integral at finite places:
\begin{equation}\label{formulaoflocalRS}
 J_p(s)=\int\limits_{NZ\backslash\GL_2}W_{f,p}(g)W_{f',p}(g)\Phi_{s,p}(g)dg
\end{equation}
where $\Phi_s$ will run over a basis of $V_E(N)$, $W_f$ is the Whittaker functional associated to $f$ and additive character $\psi$, and $W_{f'}$ is associated to $f'$ and $\psi^-$.
Then Theorem \ref{thmoffirstineq} will follow from the strategy stated in the beginning of this section.

We will pick a basis for $V_E(N)$ as follows: The local component of $\Phi_s$ at $p$ is either spherical, or supported on $B\zxz{1}{0}{p^{i}}{1}K_0(p^{e_p})$ for $i<e_p$. Note that the case when the local component is spherical is already covered in \cite{PPA14}. So we only need to consider the latter case here.
\begin{prop}\label{propRankinSelbergbound}
Locally let $\pi_i$ for $i=1,2$ be highly ramified representations of same level $c$ with trivial central characters. Let $W_1$ be the Whittaker functional associated to a newform of $\pi_1$ and additive character $\psi$.  Let $W_2$ be the Whittaker functional associated to a newform of $\pi_2$ and additive character $\psi^-$. 
Let $\Phi_s$ be a function from $\text{Ind}_B^{\GL_2}(\chi|\cdot|^{s-1/2},\chi^{-1}|\cdot|^{-s+1/2})$ with $\chi$ unitary, supported on 
$B\zxz{1}{0}{p^i}{1}K_0(p^{e_p})$ for $i<e_p$. Further suppose that $c>2e_p$. Then
\begin{equation}
 |J_p(s)|=|\int\limits_{NZ\backslash\GL_2}W_1(g)W_2(g)\Phi_{s}(g)dg|\leq \frac{p-1}{p+1}p^{-c/2}
\end{equation}
when $\Re(s)=1/2$.
\end{prop}
\begin{rem}
By the theory of newforms and oldforms in \cite{Ca73}, it would be automatic that the level of $\chi\leq \min\{i,e_p-i\}.$
 \end{rem}
\begin{rem}
  As we will only care about asymptotic behaviors, the assumption that $c>2e_p$ is reasonable.
\end{rem}

\begin{proof}
First note that 
\begin{equation}
 B\zxz{1}{0}{p^i}{1}K_0(p^{e_p})=B\zxz{1}{0}{p^i}{1}K_0(p^{j})
\end{equation}
for any $j>i$.

 As $\Phi_s$ is supported on 
$B\zxz{1}{0}{p^i}{1}K_0(p^{e_p})=B\zxz{1}{0}{p^i}{1}K_0(p^{c})$ for $i<e_p$, we have directly that
\begin{align}
 J_p(s)=\frac{p-1}{(p+1)p^{i}}\int W_1(\zxz{\alpha}{0}{0}{1}\zxz{1}{0}{p^i}{1})W_2(\zxz{\alpha}{0}{0}{1}\zxz{1}{0}{p^i}{1})\chi(\alpha)|\alpha|^{s-1}d^*\alpha.
\end{align}
Now we can apply part (3) of Proposition \ref{propWiofram}, which implies that both $W_1$ and $W_2$ taking values on $\zxz{\alpha}{0}{0}{1}\zxz{1}{0}{p^i}{1}$ are 
supported at $v(\alpha)=2i-c$, consisting of level $c-i$ components with coefficients of $L^2$ norm 1.
Then one just has to apply Cauchy-Schwartz inequality, 
and the easy fact that 
\begin{equation}
 ||\alpha|^{s-1}|=p^{i}p^{-c/2}
\end{equation}
when $v(\alpha)=2i-c$ and $\Re(s)=1/2$.
\end{proof}

\begin{rem}
 This inequality is actually better than what one could get when $\Phi_s$ is spherical locally. In particular one can argue similarly from here on to prove Theorem \ref{thmoffirstineq} as in \cite{PPA14}.
\end{rem}

\section{The second inequality}
\subsection{Fourier coefficient of Eisenstein series of level $N$}
We first study the Fourier coefficients for level $N$ Eisenstein series $E_\mathfrak{a}(\sigma_\mathfrak{b} z,s)$ by using its adelization discussed in the previous section.
Recall that
\begin{equation}\label{formulaofclassicalFourier2}
 E_\mathfrak{a}(\sigma_{\mathfrak{b}}z,s)=\delta_{\mathfrak{ab}}y^s+\varphi_\mathfrak{ab}(s)y^{1-s}+\sum\limits_{n\neq 0}\varphi_\mathfrak{ab}(n,s)\kappa_{s-1/2}(ny)e^{2\pi i nx}.
\end{equation}

\begin{lem}\label{lemofasymofconstant}
When $t\rightarrow \infty$,
\begin{equation}
\varphi_\mathfrak{ab}(1/2+it)=O(1).
\end{equation}
\end{lem}
\begin{proof}
According to Corollary \ref{corofadelization},
\begin{equation}
(\frac{N}{c c'})^{s}E_\mathfrak{a}(z,s)=\sum\limits_{\Phi_s}a_{\Phi_s}\sum\limits_{\gamma\in B(\Q)\backslash \GL_2(\Q)}\Phi_s(\gamma g_z)
\end{equation}
for proper coefficients $a_{\Phi_s}$, and $\Phi_s$ runs over a basis of $V_E(N)$. Then
\begin{align}
(\frac{N}{c c'})^{s}E_\mathfrak{a}(\sigma_\mathfrak{b} z,s)=\sum\limits_{\Phi_s}a_{\Phi_s}\sum\limits_{\gamma\in B(\Q)\backslash \GL_2(\Q)}\Phi_s(\gamma g_z').
\end{align}
Here $g_z'=(\zxz{y}{x}{0}{1},\sigma_\mathfrak{b}^{-1}, \sigma_\mathfrak{b}^{-1}, \sigma_\mathfrak{b}^{-1}, \cdots)$, and we have used that each $\sum\limits_{\gamma\in B(\Q)\backslash \GL_2(\Q)}\Phi_s(\gamma g)$ is left $\GL_2(\Q)-$ invariant.

Using the Fourier expansion for adelic Eisenstein series as in
subsection \ref{subsecofAdelization}, and comparing it with (\ref{formulaofclassicalFourier2}), we have
\begin{equation}\label{formula5.1}
(\frac{N}{c c'})^{s}\varphi_\mathfrak{aa}(s)=\sum\limits_{\Phi_s}a_{\Phi_s}\int\limits_{\A}\Phi_{s}(\omega\zxz{1}{x}{0}{1}g_z')dx,
\end{equation}

\begin{equation}\label{formula5.2}
(\frac{N}{c c'})^{s}\varphi_\mathfrak{aa}(n,s)\kappa_{s-1/2}(ny)e^{2\pi i nx}=\sum\limits_{\Phi_s}a_{\Phi_s}W(\zxz{n}{0}{0}{1}g_z').
\end{equation}
Note that $(\frac{N}{c c'})^{s}$ and $a_{\Phi_s}$ are negligible for asymptotic behavior, and for all $p\nmid N$, the corresponding local integral in (\ref{formula5.1}) will be the same as in $N=1$ case. This is because $\Phi_s$ is unramified there and $\sigma_\mathfrak{b}^{-1}$ belongs to the local maximal compact subgroup. So to prove this Lemma, it would be sufficient to prove that the normalized local integrals at $p|N$ are bounded for $\Re(s)=1/2$.

From now on we will focus on the local calculations. We pick the basis for $V_E(N)$ slightly differently from the last section, so that the local component at $p$ for $\Phi_s$ is either spherical, or only supported at $B\zxz{1}{0}{p^i}{1}K_0(p^{e_p})$ for $0<i\leq e_p$. We further suppose that $\Phi_{s,p}$ is normalized so that $\Phi_{s,p}(\zxz{1}{0}{p^i}{1})=1$. When $p\nmid N$, $\Phi_{s,p}$ is locally spherical, and $\Re(s)$ large enough, we have
\begin{equation}\label{formula5.3unramified}
\int\limits_{\Q_p}\Phi_{s}(\omega\zxz{1}{x}{0}{1}\sigma_\mathfrak{b}^{-1})dx=\int\limits_{\Q_p}\Phi_{s}(\omega\zxz{1}{x}{0}{1})dx=\frac{1-\frac{1}{p^{2s}}}{1-\frac{1}{p^{2s-1}}}.
\end{equation}
The numerator is always bounded from above and below on $\Re(s)=1/2$. The denominator could be zero for certain $s$ values.

In general let $\sigma_\mathfrak{b}^{-1}=\zxz{d_\mathfrak{b}^{-1}}{0}{0}{1}\zxz{d}{-b}{-c}{a}$ with $ad-bc=1$ and $c|N$. Note that
\begin{equation}
\omega \zxz{1}{x}{0}{1}\zxz{d_\mathfrak{b}^{-1}}{0}{0}{1}=\zxz{1}{0}{0}{d_\mathfrak{b}^{-1}}\omega \zxz{1}{d_\mathfrak{b}x}{0}{1}.
\end{equation}

When $\Phi_{s,p}$ is spherical, 
\begin{equation}
\int\limits_{\Q_p}\Phi_{s}(\omega\zxz{1}{x}{0}{1}\sigma_\mathfrak{b}^{-1})dx=|d_\mathfrak{b}|^{s-1}\frac{1-\frac{1}{p^{2s}}}{1-\frac{1}{p^{2s-1}}},
\end{equation}
which differs from (\ref{formula5.3unramified}) by a bounded factor $|d_\mathfrak{b}|^{s-1}$ on $\Re(s)=1/2$.

When $\Phi_{s,p}$ is supported at $B\zxz{1}{0}{p^i}{1}K_0(p^{e_p})$ for $0<i\leq e_p$,
\begin{equation}
\int\limits_{\Q_p}\Phi_{s}(\omega\zxz{1}{x}{0}{1}\sigma_\mathfrak{b}^{-1})dx
=\chi(d_\mathfrak{b})|d_\mathfrak{b}|^{s-1}\int\limits_{\Q_p}\Phi_{s}(\zxz{-c}{a}{cx-d}{b-ax})dx.
\end{equation}
When $v(x)\leq -2e_p$, $\zxz{-c}{a}{cx-d}{b-ax}$ is not in the support if $v(c)\neq i$. Otherwise
\begin{equation}
\zxz{-c}{a}{cx-d}{b-ax}=\zxz{\frac{1}{b-ax}}{-a\frac{d-cx}{b-ax}p^{-i}}{0}{-(d-cx)p^{-i}}\zxz{1}{0}{p^i}{1}\zxz{1}{0}{0}{-\frac{b-ax}{d-cx}p^i}.
\end{equation}
Then be definition,
\begin{equation}
|\Phi_s(\zxz{-c}{a}{cx-d}{b-ax})|=|\chi^{-1}(-acx^2p^{-i})p^{2v(x)s}|\leq p^{v(x)}.
\end{equation}
An integral over $v(x)\leq -2e_p$ would then be bounded by
\begin{equation}
\frac{p^{-(2s-1)2e_p}(1-p^{-1})}{1-\frac{1}{p^{2s-1}}}.
\end{equation}
Note that the numerator is bounded on $\Re(s)=1/2$. The denominator will be zero at certain points on $\Re(s)=1/2$, but it will be cancelled after normalized by (\ref{formula5.3unramified}).

\noindent Now the part $v(x)>-2e_p$ has finite volume. It would then be enough to show that $\Phi_s(\zxz{-c}{a}{cx-d}{b-ax})$ is bounded on $v(x)>-2e_p$ and $\Re(s)=1/2$. 

The only possible singularity would come from when $v(cx-d)$ is very large and/or $v(b-ax)$ is very large. Note that if $v(cx-d)>0$, then $x=\frac{d}{c}+pu $ for $v(u)\geq 0$. Then $b-ax=b-\frac{ad}{c}-pau=-\frac{1}{c}-pau$ is of valuation $\leq 0$. This means 
$v(cx-d)$ and $v(b-ax)$ can't both be large. When $v(cx-d)$ is large, $\Phi_s(\zxz{-c}{a}{cx-d}{b-ax})$ is clearly bounded. On the other hand, if $v(b-ax)$ is large, $\zxz{-c}{a}{cx-d}{b-ax}$ won't even be supported on $B\zxz{1}{0}{p^i}{1}K_0(p^{e_p})$ for $0<i\leq e_p$.

\end{proof}

Let $\tau(n)$ be the function counting the divisors of $n$.
\begin{lem}\label{lemofboundFourierEis}
For $\phi=E_\mathfrak{a}(\sigma_\mathfrak{b}z,s)$ and $\Re(s)=1/2$, if we write 
\begin{equation}
\phi=\delta_{\mathfrak{ab}}y^s+\varphi_\mathfrak{ab}(s)y^{1-s}+\frac{1}{\xi(2s)}\sum\limits_{n\neq 0}\frac{\lambda_{\phi,s}(n)}{\sqrt{|n|}}\kappa_{s-1/2}(ny)e^{2\pi i nx},
\end{equation}
then 
\begin{equation}
|\lambda_{\phi,s}(n)|<<_{ N}\tau(n).
\end{equation}
\end{lem}
\begin{proof}
When $N=1$, 
\begin{equation}
|\lambda_{\phi,s}(n)|=|\sum\limits_{ab=n}(\frac{a}{b})^{s-1/2}|\leq\tau(n).
\end{equation}

So the claim is clear in this case. In general, by the proof of the last lemma, 
\begin{equation}
(\frac{N}{c c'})^{s}\frac{1}{\xi(2s)}\frac{\lambda_{\phi,s}(n)}{\sqrt{|n|}} \kappa_{s-1/2}(ny)e^{2\pi i nx}=\sum\limits_{\Phi_s}a_{\Phi_s}W(\zxz{n}{0}{0}{1}g_z').
\end{equation}
Recall that we normalized the local Whittaker functional such that
\begin{equation}
W_\infty(\zxz{a}{0}{0}{1}g_z)= \frac{1}{\xi(2s)} \kappa_{s-1/2}(ay)e^{2\pi i a x},
\end{equation}
and
\begin{equation}
W_p(1)=1
\end{equation}
for all finite prime $p$.

Write $n=\prod\limits_{p|n}p^{n_p}$. It would then be enough to show that 
\begin{equation}
W_p(\zxz{n}{0}{0}{1}\sigma_\mathfrak{b}^{-1})<<_{ N}p^{-\frac{1}{2}n_p}\tau(p^{n_p})
\end{equation}
for Whittaker functions of the local component of $\Phi_s$ which runs over a basis of $V_E(N)$. We will pick the basis as follows: the local component of $\Phi_s$ is either a new form, or a translate of new form by $\zxz{p^j}{0}{0}{1}$.
 
 For locally unramified representations, the local component of $\Phi_s$ is then a $\zxz{p^j}{0}{0}{1}$ translate of spherical element. 
 Let $W_{p,0}$ be the  Whittaker functional of the spherical element as given by (\ref{formulasphericalWhittaker}). 
 Then it's clear that
 \begin{equation}
  |W_{p,0}(\zxz{n}{0}{0}{1})|\leq p^{-\frac{1}{2} n_p}\tau(p^{n_p}).
 \end{equation}
By the Iwasawa decomposition, the translate by $\sigma_\mathfrak{b}^{-1}\zxz{p^j}{0}{0}{1}$ amount to a fixed shift in the valuation for $n_p$. So
\begin{equation}
  |W_{p,0}(\zxz{n}{0}{0}{1}\sigma_\mathfrak{b}^{-1}\zxz{p^j}{0}{0}{1})|<<_N p^{-\frac{1}{2} n_p}\tau(p^{n_p}).
\end{equation}

When $\Phi_s$ belongs to an induced representation from ramified Hecke characters, still let $W_{p,0}$ be the local Whittaker functional associated to the new form of the corresponding local representation. Again the 
Iwasawa decomposition (more precisely Lemma \ref{Iwasawadecomp}) implies that the translate by $\sigma_\mathfrak{b}^{-1}\zxz{p^j}{0}{0}{1}$  will give a fixed shift in $n$ locally and also decide which double $B-$ $K_0(p^{e_p})$ coset that
$\zxz{n}{0}{0}{1}\sigma_\mathfrak{b}^{-1}\zxz{p^j}{0}{0}{1}$ belongs to. This means we only care about $W_{p,0}^{(i)}$ for some fixed $i$. Then the Lemma follows from Proposition \ref{propWiofram} and Lemma \ref{lemofgrowthofWmidlevel}. (Note that we can pick $\alpha=0$ for unitary Eisenstein series in Lemma \ref{lemofgrowthofWmidlevel}.)

\end{proof}
\subsection{proof of Theorem \ref{thmofsecondineq}}
Theorem \ref{thmofsecondineq} turns out to be easier to generalize. We shall briefly follow the proof as in  \cite{Ho10} \cite{Ne11} \cite{PPA14}, then focus on the difference.

Let $\phi$ be either a Maass eigencuspform or an incomplete Eisenstein series of level $N$. Let $f$ be a  holomorphic newform of weight $k\in 2\N$ and level $q$, where $N|q$.
Let $Y\geq 1$ be a parameter to be chosen later, and let $h\in C_c^\infty(\R^+)$ be a compactly supported everywhere nonnegative test function whose Mellin transform is $\hat{h}$ and $\hat{h}(1)=\mu(1)$.
Let $h_Y$ be the function $y\mapsto h(Yy)$.

Apply $\mu_f$ to  (\ref{spectrumdecompofincompleteEis}) where the full level incomplete Eisenstein series is chosen to be $E(z,h_Y)$. We will then get

\begin{equation}
 Y\mu_f(\phi)=\mu_f(E(z,h_Y)\phi(z))-\frac{1}{2\pi i}\int\limits_{(1/2)}\hat{h}(s)Y^s\mu_f(E(z,s)\phi(z)ds.
\end{equation}
Same argument as in \cite{Ne11} implies that
\begin{equation}
 \frac{1}{2\pi i}\int\limits_{(1/2)}\hat{h}(s)Y^s\mu_f(E(z,s)\phi(z)ds<<_\phi Y^{1/2}\mu_f(1),
\end{equation}
as the only information about $\phi$ used is its rapid decay along cusps.
Then a standard unfolding technique gives
\begin{align}
 \mu_f(E(z,h_Y)\phi(z))&=\sum\limits_{\tau\in \Gamma_\infty\backslash\SL_2(\Z)/\Gamma_0(q)}\int\limits_{\tau^{-1}\Gamma_\infty\tau\cap \Gamma_0(q)\backslash \H} h_Y(\Im(\tau z))\phi(z)|f|^2(z)y^k\frac{dxdy}{y^2}\\
 &=\sum\limits_{\tau\in \Gamma_\infty\backslash\SL_2(\Z)/\Gamma_0(q)}\int\limits_{\Gamma_\infty\cap \tau\Gamma_0(q)\tau^{-1}\backslash \H} h_Y(\Im(z))\phi(\tau^{-1}z)|f|^2(\tau^{-1}z)\Im(\tau^{-1}z)^k\frac{dxdy}{y^2}\notag\\
 &=\sum\limits_{\mathfrak{a}\in \mathcal{C}}\int\limits_{y=0}^{\infty}\int\limits_{x=0}^{d_\mathfrak{a}} h_Y(\Im(z))\phi(\tau_\mathfrak{a}z)|f|^2(\tau_\mathfrak{a}z)\Im(\tau_\mathfrak{a}z)^k\frac{dxdy}{y^2}\notag\\
 &=\sum\limits_{\mathfrak{a}\in \mathcal{C}}\int\limits_{y=0}^{\infty}\int\limits_{x=0}^{1} h_Y(\Im(d_\mathfrak{a}z))\phi(\tau_\mathfrak{a}\zxz{d_\mathfrak{a}}{0}{0}{1}z)|f|^2(\sigma_\mathfrak{a}z)\Im(\sigma_\mathfrak{a}z)^k\frac{dxdy}{y^2}\notag
 \end{align}
Here $\mathfrak{a}$ is considered a cusp for $\Gamma_0(q)$. Let $d_\mathfrak{a}'$ and $\sigma_\mathfrak{a}'$ be the width and scaling matrix for $\mathfrak{a}$ when considered as a cusp for $\Gamma_0(N)$. Then $d_\mathfrak{a}'|d_\mathfrak{a}$. If $N=1$, $d_\mathfrak{a}'=1$, we have a single Fourier expansion of $\phi$ along cusps. But in general $d_\mathfrak{a}'$ may not be 1, and we can have several Fourier expansions along different cusps.
This is the difference between our case and previous papers.

Suppose we have the Fourier expansion
\begin{equation}
\phi(\sigma_\mathfrak{a}'z)=\sum\limits_{l\in\Z}\phi_{\mathfrak{a},l}(y)e^{2\pi i lx}.
\end{equation}
Let $\tilde{d_\mathfrak{a}}=d_\mathfrak{a}/d_\mathfrak{a}'$. Set
\begin{align*}
\mathcal{S}_0&=\sum\limits_{\mathfrak{a}\in \mathcal{C}}\int\limits_{y=0}^{\infty}h_Y(d_\mathfrak{a}y)\int\limits_{x=0}^{1} \phi_{\mathfrak{a},0}(\tilde{d_\mathfrak{a}}y)|f|^2(\sigma_\mathfrak{a}z)\Im(\sigma_\mathfrak{a}z)^k\frac{dxdy}{y^2},\\
\mathcal{S}_{0,Y^{1+\epsilon}}&=\sum\limits_{\mathfrak{a}\in \mathcal{C}}\int\limits_{y=0}^{\infty}h_Y(d_\mathfrak{a}y)\int\limits_{x=0}^{1} \sum\limits_{0<|l|<Y^{1+\epsilon}}\phi_{\mathfrak{a},l}(\tilde{d_\mathfrak{a}}y)|f|^2(\sigma_\mathfrak{a}z)\Im(\sigma_\mathfrak{a}z)^k e^{2\pi i l\tilde{d_\mathfrak{a}}x}\frac{dxdy}{y^2},\\
\mathcal{S}_{\geq Y^{1+\epsilon}}&=\sum\limits_{\mathfrak{a}\in \mathcal{C}}\int\limits_{y=0}^{\infty}h_Y(d_\mathfrak{a}y)\int\limits_{x=0}^{1} \sum\limits_{|l|>Y^{1+\epsilon}}\phi_{\mathfrak{a},l}(\tilde{d_\mathfrak{a}}y)|f|^2(\sigma_\mathfrak{a}z)\Im(\sigma_\mathfrak{a}z)^ke^{2\pi i l\tilde{d_\mathfrak{a}}x}\frac{dxdy}{y^2}.
\end{align*}
So
\begin{equation}
\mu_f(E(z,h_Y)\phi(z))=\mathcal{S}_0+\mathcal{S}_{0,Y^{1+\epsilon}}+\mathcal{S}_{\geq Y^{1+\epsilon}}.
\end{equation}
\begin{lem}(the main term)
$\mathcal{S}_0$ is $0$ when $\phi$ is Maass eigencuspform. If $\phi$ is incomplete Eisenstein series, 
\begin{equation}
\mathcal{S}_0=Y\mu_f(1)(\frac{\mu(\phi)}{\mu(1)}+O_\phi(\frac{1+R_f(qk)}{Y^{1/2}})),
\end{equation}
where 
\begin{equation}
R_f(x)=\frac{x^{-1/2}}{L(f,Ad,1)}\int\limits_{\R}|\frac{L(f,Ad,1/2+it)}{(1+|t|)^{10}}|dt
\end{equation}
is independent of $\phi$.
\end{lem}
\begin{proof}
The first part is clear. For incomplete Eisenstein series $\phi$ of level 1, this is Lemma 3.6 of \cite{Ne11}. The argument there used that for $y\asymp 1/Y$, 
\begin{equation}\label{phi0y}
\phi_{\mathfrak{a},0}(y)=\mu(\phi)/\mu(1)+O_\phi(Y^{-1/2}).
\end{equation}
By (\ref{formulaEisconstantterm}) and (\ref{formula4.1}) we know that for $\phi=E_\mathfrak{a}(z,\Psi)$ for a compactly supported function $\Psi$ on $\R^*_+$,
\begin{equation}
\phi_{\mathfrak{a},0}(y)=\mu(\phi)/\mu(1)+\frac{1}{2\pi i}\int\limits_{(1/2)}\hat{\Psi}(s)(\delta_{\mathfrak{ab}}y^s+\varphi_\mathfrak{ab}(s)y^{1-s})ds.
\end{equation}
If $\phi$ is of level 1, (\ref{phi0y}) follows from that $\hat{\Psi}$ is rapidly decreasing and $\varphi_\mathfrak{ab}(s)=M(s)$ is always of norm 1 on $\Re(s)=1/2$. In general it follows from Lemma \ref{lemofasymofconstant}. The rest argument would be the same as in \cite{Ne11}. 
\end{proof}
\begin{lem}(Trivial error term)
\begin{equation}
 \mathcal{S}_{\geq Y^{1+\epsilon}}<<_{\phi,\epsilon}Y^{-10}\mu_f(1).
\end{equation}
\end{lem}
\begin{proof}
 The original proof in \cite{Ne11} made use of a bound for the sum of Fourier coefficients for $\phi$, which in our case follows directly from 
 Corollary \ref{corofboundFouriercuspidal} and Lemma \ref{lemofboundFourierEis}.
\end{proof}
Now we consider the main error term $\mathcal{S}_{0,Y^{1+\epsilon}}$.
Recall by (\ref{formulaofholoFourier}),
\begin{equation}
 f|_k\sigma_\mathfrak{a}(z)=y^{-k/2}\sum\limits_{n\in \N}\frac{\lambda_\mathfrak{a}(n)}{\sqrt{n}}\kappa_f(ny)e^{2\pi inx}
\end{equation}
for any cusp $\mathfrak{a}$ and $\kappa_f(y)=y^{k/2}e^{-2\pi y}$.
Note that $|f|^2(\sigma_\mathfrak{a}z)\Im(\sigma_\mathfrak{a}z)^k=|f|_k\sigma_\mathfrak{a}|^2(z)y^k$. Then 
\begin{equation}
|f|^2(\sigma_\mathfrak{a}z)\Im(\sigma_\mathfrak{a}z)^k=\sum\limits_{m,n\in \N}\frac{\lambda_\mathfrak{a}(n)\overline{\lambda_\mathfrak{a}(m)}}{\sqrt{nm}}\kappa_f(ny)\kappa_f(my)e^{2\pi i(n-m)x}.
\end{equation}
We first focus on the case when $\phi$ is a Maass eigencuspform, then we can write $\phi_{\mathfrak{a},l}$ more explicitly as
\begin{equation}
\phi_{\mathfrak{a},l}(y)=\frac{\lambda_{\phi,\mathfrak{a}}(l)}{\sqrt{l}}\kappa_{ir}(ly).
\end{equation}
Define
\begin{equation}
I_\phi(l,n,x)=(mn)^{-1/2}\int\limits_{0}^{\infty}h(xy)\kappa_{ir}(ly)\kappa_f(my)\kappa_f(ny)\frac{dy}{y^2}, m=n+l.
\end{equation}
Then
\begin{align}\label{formulaerrorterm5.1}
\mathcal{S}_{0,Y^{1+\epsilon}}&=\sum\limits_{\mathfrak{a}\in \mathcal{C}}\int\limits_{y=0}^{\infty}h_Y(d_\mathfrak{a}y)\int\limits_{x=0}^{1} \sum\limits_{0<|l|<Y^{1+\epsilon}}\phi_{\mathfrak{a},l}(\tilde{d_\mathfrak{a}}y)|f|^2(\sigma_\mathfrak{a}z)\Im(\sigma_\mathfrak{a}z)^ke^{2\pi i l\tilde{d_\mathfrak{a}}x}\frac{dxdy}{y^2}\\
&=\sum\limits_{\mathfrak{a}\in \mathcal{C}} \sum\limits_{0<|l|<Y^{1+\epsilon}}\frac{\lambda_{\phi,\mathfrak{a}}(l)}{\sqrt{l}}\sum\limits_{n\in\N,m=n+\tilde{d_\mathfrak{a}}l}\lambda_\mathfrak{a}(n)\overline{\lambda_\mathfrak{a}(m)}I_\phi(\tilde{d_\mathfrak{a}}l,n,d_\mathfrak{a}Y)\notag
\end{align}
For simplicity, let $d_c=\frac{[q,c_\mathfrak{a}^2]}{c_\mathfrak{a}^2}=d_\mathfrak{a}$ for $\mathfrak{a}\in\mathcal{C}[c]$.
When $N=1$, we will get
\begin{align}
|\mathcal{S}_{0,Y^{1+\epsilon}}|&=|\sum\limits_{\mathfrak{a}\in \mathcal{C}} \sum\limits_{0<|l|<Y^{1+\epsilon}}\frac{\lambda_{\phi}(l)}{\sqrt{l}}\sum\limits_{n\in\N,m=n+d_\mathfrak{a}l}\lambda_\mathfrak{a}(n)\overline{\lambda_\mathfrak{a}(m)}I_\phi(d_\mathfrak{a}l,n,d_\mathfrak{a}Y)|\\
 &=\sum\limits_{0<|l|<Y^{1+\epsilon}}|\sum\limits_{c|q}\frac{\lambda_{\phi}(l)}{\sqrt{l}}\sum\limits_{n\in\N,m=n+d_cl}I_\phi(d_cl,n,d_cY)\sum\limits_{\mathfrak{a}\in \mathcal{C}[c]} \lambda_\mathfrak{a}(n)\overline{\lambda_\mathfrak{a}(m)}|\notag\\
&\leq \sum\limits_{0<|l|<Y^{1+\epsilon}}\sum\limits_{c|q}\sharp \mathcal{C}[c]\frac{|\lambda_{\phi}(l)|}{\sqrt{l}}\sum\limits_{n\in\N,m=n+d_cl}|I_\phi(d_cl,n,d_cY)||\lambda_{[c]}(n)\lambda_{[c]}(m)|.\label{formulareduced}
 \end{align}
Here $\lambda_{[c]}(n)$ is as defined in (\ref{defofaveragelambda}), and the last inequality follows simply from Cauchy-Schwartz inequality.
Then it is proven in \cite{Ne11}, \cite{PPA14} that
\begin{equation}\label{5.1}
 I_\phi(l,n,x)<<_A \frac{\Gamma(k-1)}{(4\pi)^{k-1}}\max\{1,\frac{\max\{m,n\}}{xk}\}^{-A}
\end{equation}
for every $A>0$,
\begin{equation}\label{5.3}
 \sum\limits_{0<|l|<Y^{1+\epsilon}}\frac{|\lambda_{\phi}(l)|}{\sqrt{|l|}}<<_{\phi,\epsilon}Y^{1/2+2\epsilon},
\end{equation}
and a bound of shifted convolution sum

\begin{equation}\label{5.2}
 \sum\limits_{n\in\N,m=n+l,\max\{m,n\}\leq x}|\lambda_{[c]}(n)\lambda_{[c]}(m)|<<_\epsilon q^{\epsilon}_\Diamond \log \log(e^eq)^{O(1)}\frac{x\prod\limits_{p\leq x}(1+2|\lambda_f(p)|/p)}{\log(ex)^{2-\epsilon}}.
\end{equation}

Combining (\ref{5.1}) (\ref{5.2}) into (\ref{formulareduced}) and summing dyadically in terms of $\max\{m,n\}$, one can get a bound for $\sum\limits_{n\in\N,m=n+d_cl}|I_\phi(d_cl,n,d_cY)||\lambda_{[c]}(n)\lambda_{[c]}(m)|$, which is actually independent of $l$. Then applying (\ref{5.3}) and  Deligne bound $|\lambda_f(p)|\leq 2$, and taking $Y$ as in \cite{PPA14} will prove Theorem \ref{thmofsecondineq} for $N=1$ case.

In general for our case, the first issue is that $\lambda_{\phi,\mathfrak{a}}(l)$ could be different for different cusps. But there is no harm to be a little loose as there are only finitely many fixed cusps for $\Gamma_0(N)$. Denote
\begin{equation}
 \lambda_{\phi,+}(l)=\sum\limits_{\text{cusps for $\Gamma_0(N)$}}|\lambda_{\phi,\mathfrak{a}}(l)|.
\end{equation}
Then by Corollary \ref{corofboundFouriercuspidal}
\begin{equation}
 \sum\limits_{0<|l|<Y^{1+\epsilon}}\frac{|\lambda_{\phi,+}(l)|}{\sqrt{|l|}}<<_{\phi,\epsilon}Y^{1/2+2\epsilon},
\end{equation}
which is the analogue of (\ref{5.3}). Note that for fixed $c|q$, $d'_\mathfrak{a}$ will also be the same for all $\mathfrak{a}\in\mathcal{C}[c]$. Then
\begin{align}
 |\mathcal{S}_{0,Y^{1+\epsilon}}|&\leq \sum\limits_{0<|l|<Y^{1+\epsilon}}\sum\limits_{c|q}\frac{|\lambda_{\phi,+}(l)|}{\sqrt{|l|}}\sum\limits_{n\in\N,m=n+\tilde{d_\mathfrak{a}}l}|I_\phi(\tilde{d_c}l,n,d_cY)|\sum\limits_{\mathfrak{a}\in \mathcal{C}[c]} |\lambda_\mathfrak{a}(n)\overline{\lambda_\mathfrak{a}(m)}|\\
 &\leq \sum\limits_{0<|l|<Y^{1+\epsilon}}\sum\limits_{c|q}\sharp\mathcal{C}[c]\frac{|\lambda_{\phi,+}(l)|}{\sqrt{|l|}}\sum\limits_{n\in\N,m=n+\tilde{d_c}l}|I_\phi(\tilde{d_c}l,n,d_cY) \lambda_{[c]}(n)\lambda_{[c]}(m)|\notag.
\end{align}
Note that the inner sum in $n$ differs from (\ref{formulareduced}) only by $l$ part, and we already know that we can get an upper bound which is independent of $l$.
Then one can argue similarly from this point on to prove Theorem \ref{thmofsecondineq} as in \cite{PPA14}.
We will not give further details. The main point here is a control for Fourier coefficients of Maass eigencuspform of level $N$ as in Corollary \ref{corofboundFouriercuspidal}.

When $\phi$ is a incomplete Eisenstein series, one can decompose it into residue spectrum and continuous spectrums as in (\ref{spectrumdecompofincompleteEis}), and proceed as in the Maass eigencuspform case. The main point will again be a control of Fourier coefficients for Eisenstein series of level $N$, which follows directly from Lemma \ref{lemofboundFourierEis}.


\end{document}